\documentclass[11pt,reqno,oneside]{article}

\usepackage{a4wide,amsthm,amsfonts}
\usepackage{latexsym}
\usepackage{amsmath}
\usepackage{amssymb}
\usepackage[numbers]{natbib}
\usepackage{mathrsfs}
\usepackage[utf8]{inputenc}
\usepackage{mathtools}
\usepackage{amsmath}
\usepackage{amsthm}
\usepackage{bbm}
\usepackage{amsfonts}
\usepackage{xcolor}
\usepackage{hyperref}
\hypersetup{
    colorlinks,
    citecolor=red,
    filecolor=black,
    linkcolor=red,
    urlcolor=black
}

\theoremstyle{plain}
\newtheorem{Thm}{Theorem}[section]
\newtheorem{Lemma}[Thm]{Lemma}

\newtheorem{Proposition}[Thm]{Proposition}
\newtheorem{Corollary}[Thm]{Corollary}
\newtheorem{Definition}[Thm]{Definition}
\newtheorem{Example}[Thm]{Example}
\newtheorem{Remark}[Thm]{Remark}

\numberwithin{equation}{section}

\newcommand{\E}{\ensuremath{\mathbb{E}\,}}
\newcommand{\R}{\ensuremath{\mathbb{R}}}
\newcommand{\N}{\ensuremath{\mathbb{N}}}

\newcommand{\normm}[1]{{\left\vert\kern-0.25ex\left\vert\kern-0.25ex\left\vert #1 
    \right\vert\kern-0.25ex\right\vert\kern-0.25ex\right\vert}}

\newcommand{\abs}[1]{\left|\left|#1\right|\right|}

\newcommand{\norm}[1]{\left\lVert#1\right\rVert}

\renewcommand{\d}{{\rm d}}                                

\usepackage{color}

\renewcommand{\Bar}[1]{\mkern 1.5mu\overline{\mkern-1.5mu#1\mkern-1.5mu}\mkern 1.5mu}
\renewcommand{\phi}{\varphi}

\allowdisplaybreaks

\providecommand{\keywords}
{
	\textbf{\textit{Keywords and phrases:}}
}
\providecommand{\amssubj}
{
	\textbf{\textit{AMS 2010 subject classification:}}
}

\author{Gergely Bodó \thanks{tomasz.kosmala@kcl.ac.uk} \qquad\qquad\qquad  Markus Riedle \thanks{markus.riedle@kcl.ac.uk} \\Department of Mathematics \\ King's College London\\ London WC2R 2LS\\ United Kingdom}

\author{Gergely Bodó \\Department of Mathematics \\ King's College  \\ London WC2R 2LS\\ United Kingdom\\ \\ gergely.bodo@kcl.ac.uk\\ 
    \and Markus Riedle\footnote{2nd affiliation: Institute of Mathematical Stochastics, Faculty of Mathematics, TU Dresden,   01062 Dresden, Germany}
     \\Department of Mathematics \\ King's College  \\ London WC2R 2LS\\ United Kingdom\\ \\ markus.riedle@kcl.ac.uk
\and Ond\v rej T\'ybl \footnote{The author was supported in part by Czech Science Foundation (GA\,\v CR) grant no./22-12790S and by the Grant schemes at Charles University, project no. CZ.02.2.69\/0.0\/0.0\/19 073/0016935. }\\     
Department of Probability and  Mathematical Statistics\\
Charles University \\
Sokolovsk\'a 49/83 186 75 Prague 8\\ 
Czech Republic\\ \\
tybl@karlin.mff.cuni.cz
 }

\title{SPDEs driven by standard symmetric $\alpha$-stable cylindrical Lévy processes: 
    existence, Lyapunov functionals and It{\^o} formula}

\begin{document}

\maketitle

\begin{abstract}
We investigate several aspects of solutions to stochastic evolution equations in Hilbert spaces driven by a standard symmetric $\alpha$-stable cylindrical noise. Similarly to cylindrical Brownian motion or Gaussian white noise, standard symmetric $\alpha$-stable noise exists only in a generalised sense in Hilbert spaces.
The main results of this work are the existence of a mild solution, long-term regularity of the solutions via Lyapunov functional approach, and an It\^{o} formula for mild solutions to evolution equations under consideration. The main tools for establishing these results are Yosida approximations and an It\^{o} formula for Hilbert space-valued semi-martingales where the martingale part is represented as an integral driven by cylindrical $\alpha$-stable noise. While these tools are standard in stochastic analysis, due to the cylindrical nature of our noise, their application requires completely novel arguments and techniques.
\end{abstract}

\begin{flushleft}
	\amssubj{60H15, 60G20, 60G52, 60H05}
	
\end{flushleft}

\section{Introduction}

Standard symmetric $\alpha$-stable distributions are the natural generalisations of Gaussian distributions for modelling random  perturbations of finite dimensional dynamical systems. They often meet various empirical requests, such as heavy tails, self-similarity and infinite variance, but are at the same time  analytically tractable and well-understood.
The importance of these models is reflected by the available vast literature on dynamical systems perturbed by random noises with $\alpha$-stable distributions in various areas such as economics, biology etc. 

In the infinite dimensional setting of modelling random perturbations of partial differential equations, much fewer results are known for systems perturbed by $\alpha$-stable distributions. 
In fact, only in the random field approach, based on the seminal work by Walsh, one can find several publications on stochastic partial differential equations (SPDEs) driven by multiplicative $\alpha$-stable noise, e.g.\ Mueller \cite{mueller1998heat}, Mytnik \cite{mytnik2002stochastic}, and more recently Chong \cite{Chong_stochastic_PDEs} and Chong et.\ al.\ \cite{Chong_Dalang_Humeau}. 
However, in the semigroup approach, following the spirit of Da Prato and Zabczyk, one can find several results for equations only with additive driving noise distributed according to an $\alpha$-stable law; see e.g.\ Brze\'zniak and Zabczyk \cite{Brzezniak_Zabczyk} and Riedle \cite{RI}. The only publication in the semigroup approach for multiplicative $\alpha$-stable perturbation is Kosmala and Riedle \cite{KR}, where however the assumptions are rather restrictive and do not correspond to the natural Lipschitz continuity and linear growth conditions. The lack of results in the semigroup approach is due to the fact that a random noise with a standard symmetric $\alpha$-stable distribution does not exist as an ordinary  Hilbert space-valued process but only in the generalised sense of Gel'fand and Vilenkin \cite{GV} or Segal \cite{S}. 

In this work, we investigate several aspects of solutions to equations of the form
\begin{equation}
    {\rm d}X(t)=\big( AX(t) + F(X(t))\big)\, {\rm d}t + G(X(t-))\, {\rm d}L(t),
    \label{introduction_equation}
\end{equation}
where $A$ is the generator of a $C_0$-semigroup in a separable Hilbert space $H$, the coefficients $F\colon H\to H$ and  $G\colon H\to\mathcal{L}_2(U,H)$ are mappings with $U$ being a separable Hilbert space, and $L$ is a standard symmetric $\alpha$-stable cylindrical process in $U$ for $\alpha\in (1,2)$.

Analogously to the standard normal distribution, standard symmetric $\alpha$-stable distributions in $\R^d$ can only be generalised to infinite dimensional spaces as cylindrical distributions. In particular, this means that the driving noise $L$ in \eqref{introduction_equation} exists only in the generalised sense;  see Schwartz \cite{Schwartz}. Since such processes do not attain values in the underlying Hilbert space, standard results for stochastic processes in infinite dimensional spaces are not applicable. Most notably, complications arise from the fact that while these processes are cylindrical semi-martingales, see Jakubowski et.\ al.\ \cite{JKFR}, they do not enjoy a semi-martingale decomposition in a cylindrical sense, since semi-martingale decompositions are not invariant under linear transformations, see Jakubowski and Riedle \cite[Re.\ 2.2]{JR}. Nevertheless, the problem of stochastic integration with respect to cylindrical L\'evy processes was solved in Jakubowski and Riedle \cite{JR} by arguments avoiding the usual   L\'evy-It{\^o} decomposition. This approach has been further developed for standard symmetric $\alpha$-stable cylindrical process by two of us in Bod\'o and Riedle \cite{BR}, which enables us to integrate predictable integrands and to derive a dominated convergence theorem for stochastic integrals.

This work comprises of 3 main results: the existence of a mild solution to Equation \eqref{introduction_equation}, a Lypunov functional approach for long-term regularity for solutions to  Equation \eqref{introduction_equation}, and 
an It{\^o} formula for mild solutions to  Equation \eqref{introduction_equation}.  The main tools for establishing these results are an It{\^o} formula for Hilbert space-valued semi-martingales driven by standard symmetric $\alpha$-stable cylindrical L\'evy noise and a Yosida approximation of solutions to Equation \eqref{introduction_equation}. While these tools are standard in stochastic analysis, due to the cylindrical nature of our noise, their application in our setting requires completely novel arguments and techniques, which we highlight in the following.

A classical It{\^o}  formula for semi-martingales in Hilbert spaces is well known and easy to derive; see e.g.\ Metivier \cite[Th.\ 27.2]{ME}. However, applying this formula often requires the identification of the martingale and bounded variation components of the process, which in the classical situation of a semi-martingale driven by an ordinary Hilbert space-valued process  can easily be obtained via the semi-martingale decomposition of the driving process. Since in our case, the driving cylindrical process does not enjoy a semi-martingale decomposition, one needs to identify the martingale part of the stochastic integral process by carrying out a deep analysis of its jump structure.

The second major tool in our work is a Yosida approximation, which is an often-utilised device in the classical situation with an ordinary Hilbert space-valued process as driving noise; see e.g.\ Peszat and Zabczyk \cite{PZ}. 
Convergence of the Yosida approximation is established by tightness arguments in the space $\mathcal{C}([0,T], L^p(\Omega,H))$ of  $p$-th mean continuous Hilbert space-valued processes for any $p<\alpha$. It turns out that the space  $\mathcal{C}([0,T], L^p(\Omega,H))$ is tailor-made for analysing equations driven by a standard symmetric $\alpha$-stable cylindrical process. The observation that the solution is continuous in the above sense, despite having discontinuous paths, lies at the heart of this paper. To the best of our knowledge, we are the first to use this in the context of SPDEs driven by cylindrical stable noise.

These two tools, the It{\^o} formula for semi-martingales driven by a standard symmetric $\alpha$-stable cylindrical process
and convergence of the Yosida approximation,  enable us to establish the 3 main results of our work. For the existence result, we use tightness of the Yosida approximation to establish existence of a mild solution to Equation \eqref{introduction_equation}. In our setting, standard methods for establishing existence of a solution, such as fix point arguments or Grönwall's lemma are not applicable, since the integral operator with a standard symmetric $\alpha$-stable integrator maps to a larger space than its domain; see Kosmala and Riedle \cite{KR} or Rosinski and Woyczynski \cite{RW}.

By following the classical approach of Ichikawa in \cite{ICH},
we demonstrate the power of the established tools by investigating the long-term regularity of the mild solution to Equation \eqref{introduction_equation} via the functional Lyapunov approach. The functional Lyapunov approach can be used to establish various regularity properties; in this work, we focus on exponential ultimate boundedness, but other quantitative properties can be investigated similarly.  
As the mild solution is not a semi-martingale, the derived It{\^o} formula for semi-martingales cannot be applied directly. However, we successfully show that the Yosida approximations are semi-martingales, and thus the It{\^o} formula can be applied to these, which immediately shows their exponential ultimate boundedness. It remains only to show that this boundedness property carries over to the limit, for which we establish convergence of the Markov generators in a suitable sense.

Mild solutions of SPDEs are not semi-martingales, and thus the classical It{\^o} formula cannot be applied. 
This lack of a powerful tool is often circumvented by a specific  It{\^o} formula for mild solutions of SPDEs. One of the first versions of such an It{\^o} formula for mild solutions can be found in  Ichikawa \cite{ICH} for the Gaussian case, 
and more recent versions in Da Prato et.\ al.\ \cite{da2019mild} for the Gaussian case and in Alberverio et.\ al.\ \cite{AGMRS} for the case of ordinary L\'evy processes.  In the last part of our work, we derive such an It{\^o} formula for mild solutions of equation \eqref{introduction_equation}  driven by a standard symmetric $\alpha$-stable cylindrical process.

We outline the structure of the paper. Selected preliminaries on standard symmetric $\alpha$-stable cylindrical processes, integration with respect to them and underlying results on equations as well as the theory of predictable compensators are collected in Section 2. In Sections 3 and 4, we identify the predictable compensator and quadratic variation of the integral process. These observations lead us directly to the Itô formula for semi-martingales driven by a standard symmetric $\alpha$-stable cylindrical process in Section 5. In Section 6, we prove existence of a mild solution under Lipschitz and boundedness conditions in the space of continuous functions, where the main result is formulated in Theorem \ref{thm:existence_mild}. In Section 7, we establish conditions for exponential ultimate boundedness in Theorem \ref{thm:criterion}. Finally, in Section 8, an Itô formula for mild solutions is proved.

\section{Preliminaries}

\subsection{Standard symmetric $\alpha$-stable cylindrical L\'evy processes}

 Let $U$ and $H$ be separable Hilbert spaces with norm $\norm{\cdot}$ and scalar product $\langle\cdot,\cdot\rangle$. By $\overline{B}_H(r)$ we denote the closed ball in $H$ with radius $r>0$ and, in the special case when $r=1$, we write $\overline{B}_H:=\overline{B}_H(1)$. 
The space of Hilbert-Schmidt operators $\Phi\colon U\to H$ is denoted by $\mathcal{L}_2(U,H)$ and equipped with the norm 
$\norm{\cdot}_{\mathcal{L}_2(U,H)}$. 

Let $S$ be a subset of $U$. For each $n \in \mathbb{N}$, elements $u_1,...,u_n \in S$ and Borel set $A \in \mathcal{B}(\mathbb{R}^n)$, we define
\[C(u_1,...,u_n;A):=\{u \in U: (\langle u,u_1 \rangle,...,\langle u,u_n \rangle)\in A\}.\]
Such sets are called cylindrical sets with respect to $S$ and the collection of all such cylindrical sets is denoted by  $\mathcal{Z}(U,S)$. It is a $\sigma$-algebra if $S$ is finite and otherwise an algebra. We write shortly $\mathcal{Z}(U)$ for $\mathcal{Z}(U,U)$.\par
A set function $\mu: \mathcal{Z}(U)\rightarrow [0, \infty]$ is called a cylindrical measure on $\mathcal{Z}(U)$ if for each finite subset $S \subseteq U$, the restriction of $\mu$ to the $\sigma$-algebra $\mathcal{Z}(U,S)$ is a $\sigma$-additive measure. A cylindrical measure is said to be a cylindrical probability measure if $\mu(U)=1$.\par

Let $(\Omega,\Sigma,P)$ be a complete probability space. We will denote by $L_P^0(\Omega,U)$ the space of equivalence classes of measurable functions $Y\colon \Omega \rightarrow U$ equipped with the topology of convergence in probability.
A cylindrical random variable $X$ in $U$ is a linear and continuous mapping $X\colon U \rightarrow L_{P}^0(\Omega,\mathbb{R})$. It defines a cylindrical probability measure $\mu_X$ by 
\begin{align*}
    \mu_X\colon  \mathcal{Z}(U) \to [0,1],\qquad
    \mu_X(Z)=P\big( (Xu_1,\dots, Xu_n)\in A\big),
\end{align*}
for cylindrical sets $Z=C(u_1, ... , u_n; A)$. The cylindrical probability measure $\mu_X$ is called the {\em cylindrical distribution} of $X$.
We define the characteristic function of the cylindrical random variable $X$ by
\[\varphi_X\colon U \rightarrow \mathbb{C}, \qquad\varphi_X(u)=E\big[e^{iXu}\big].\]
Let $T\colon U\to H$ be a linear and continuous operator. By defining
\begin{align*}
    TX\colon H\to L_{P}^0(\Omega,\mathbb{R}), \qquad (TX)h=X(T^\ast h),
\end{align*}
we obtain a cylindrical random variable on $H$. In the special case when $T$ is a Hilbert-Schmidt operator and hence $0$-Radonifying by \cite[Th.\ VI.5.2]{VTC}, it follows from  \cite[Pr.\ VI.5.3]{VTC} that the cylindrical random variable $TX$ is induced by a genuine random variable $Y\colon \Omega\to H$, that is 
$(TX)h =\langle Y,h \rangle$  for all $h\in H$. 

A family $(L(t):t\geq 0)$ of cylindrical random variables $L(t)\colon U \rightarrow L_{P}^0(\Omega,\mathbb{R})$ is called a cylindrical $\left(\mathcal{F}_t\right)$-Lévy process if for each $n \in \mathbb{N}$ and $u_1,...,u_n \in U$, the stochastic process
$\big(\big(L(t)u_1,...,L(t)u_n\big): t \geq 0\big)$
is an $\left(\mathcal{F}_t\right)$-Lévy process in $\mathbb{R}^n$ and the filtration $\left(\mathcal{F}_t\right)_{t\geq 0}$ satisfies the usual conditions.
We denote by $\mathcal{Z}_*(U)$ the collection
\[\big\{\{u \in U: (\langle u,u_1 \rangle,...,\langle u,u_n \rangle)\in B\}:n \in \mathbb{N},u_1,...,u_n \in U,B\in \mathcal{B}(\mathbb{R}^n\setminus\{0\})\big\}\]
of cylindrical sets, which  forms an algebra of subsets of $U$. For fixed $u_1,...,u_n \in U$, let $\lambda_{u_1,\dots, u_n}$ be the L\'evy measure of $\big((L(t)u_1,...,L(t)u_n ): t \geq 0\big)$. Define a function 
$\lambda\colon  \mathcal{Z}_*(U) \rightarrow [0,\infty]$ by 
\begin{align*}
\lambda(C):=\lambda_{u_1,\dots, u_n}(B)\quad\text{for} \quad
C=\{u\in U:\, (\langle u,u_1 \rangle,...,\langle u,u_n \rangle)\in B\},
\end{align*}
for $B\in\mathcal{B}(\R^n)$. It is shown in \cite{AR} that $\lambda$ is well defined. The set function $\lambda$ is called the cylindrical L\'evy measure of $L$.

In this paper, we restrict our attention to standard symmetric $\alpha$-stable cylindrical Lévy processes for $\alpha\in (1,2)$, which we simply call $\alpha$-stable cylindrical L\'evy processes in the sequel. These are 
cylindrical Lévy processes with characteristic function $\phi_{L(t)}(u)=\exp(-t\norm{u}^\alpha)$ for each $t\ge 0$ and $u\in U$. Let $(e_k)_{k\in\N}$ be an orthonormal basis of $U$. 

The L\'evy measure $\lambda$ of an $\alpha$-stable cylindrical Lévy process for $\alpha\in (1,2)$ satisfies
\begin{align}
\left(\lambda\circ\Phi^{-1}\right)(\overline{B}_H^c)\leq c_\alpha\norm{\Phi}_{\mathcal{L}_2(U,H)}^\alpha, \quad \Phi\in\mathcal{L}_2(U,H)
\label{eq:estimate_ball_alpha}
\end{align}
for some $c_\alpha\in (0,\infty)$ depending only on $\alpha$ where $\overline{B}_H^c$ denotes the complement of the unit ball $\{h\in H:\, \norm{h}< 1\}$; see  \cite[Le.\ 1]{KR}. This enables us to conclude 
the following technical Lemma: 
\begin{Lemma}
Let $\lambda$ be the cylindrical L\'evy measure  of an $\alpha$-stable cylindrical Lévy process for $\alpha\in (1,2)$. 
For every $m\in\mathbb{N}$ there exists $d_\alpha^m<\infty$, depending only on $\alpha$ and $m$, such that
\begin{align}
\int_{\overline{B}_H(1/m)}{\abs{h}^2\,(\lambda\circ\Phi^{-1})(\d h)}+\int_{\overline{B}_H(m)^c}{\abs{h}\,(\lambda\circ\Phi^{-1})(\d h)}\leq d_\alpha^m\abs{\Phi}_{\mathcal{L}_2(U, H)}^\alpha
\label{estimate_alpha_norm}
\end{align}
for all $\Phi\in\mathcal{L}_2(U, H)$. Moreover, we have $\lim_{m \rightarrow \infty}d_\alpha^m=0$.
\end{Lemma}
\begin{proof}
Let $m \in \mathbb{N}$ be fixed. We approximate the integrand of the first integral in \eqref{estimate_alpha_norm} by 
\[ f_{m,n}\colon \overline{B}_H(\tfrac{1}{m}) \rightarrow \mathbb{R}, 
\qquad 
f_{m,n}(h):=\sum_{i=0}^{m2^n-1}\left(\frac{i}{m2^n}\right)^2 \mathbbm{1}_{(\frac{i}{m2^n},\frac{i+1}{m2^n}]}(\norm{h}).\]
Since $\lambda\circ\Phi^{-1}$ is a genuine $\alpha$-stable measure in $H$,  we have
for each $r>0$ that 
\begin{align}
    (\lambda \circ \Phi^{-1})(\overline{B}_H^c(r))=r^{-\alpha}(\lambda \circ \Phi^{-1})(\overline{B}_H^c)
    \label{eq:stable_ball_radius};
\end{align}
see \cite[Th.\ 6.2.7]{LI}. This enables us to conclude  for each $n\in\mathbb{N}$ that
\begin{align*}
    \int_{\overline{B}_H(1/m)}f_{m,n}(h)&\,(\lambda \circ \Phi^{-1})(\d h)\\
    =& \frac{1}{m^2}(\lambda \circ \Phi^{-1})(\overline{B}_H^c)\sum_{i=0}^{m2^n-1}\left(\frac{i}{2^n}\right)^2 \left(\left(\frac{i}{m2^n}\right)^{-\alpha}-\left(\frac{i+1}{m2^n}\right)^{-\alpha}\right)\\
    =& m^{\alpha-2}(\lambda \circ \Phi^{-1})(\overline{B}_H^c) \int_0^1\left(\sum_{i=0}^{m2^n-1}\left(\frac{i}{2^n}\right)^2 \mathbbm{1}_{\left(\frac{i}{2^n},\frac{i+1}{2^n}\right]}(r)\right)\alpha r^{-(\alpha+1)}\, {\rm d}r. 
\end{align*}
The monotone convergence theorem implies
\begin{align*}
    \lim_{n \rightarrow \infty}\int_{\overline{B}_H(1/m)}f_{m,n}(h)\,(\lambda \circ \Phi^{-1})(\d h)
    &=m^{\alpha-2}\frac{\alpha}{2-\alpha}(\lambda \circ \Phi^{-1})(\overline{B}_H^c). 
\end{align*}
Since another application of the monotone convergence theorem shows
\begin{equation*}
    \lim_{n \rightarrow \infty}\int_{\overline{B}_H(1/m)}f_{m,n}(h)\, (\lambda \circ \Phi^{-1})(\d h)=\int_{\overline{B}_H(1/m)}\norm{h}^2\, (\lambda \circ \Phi^{-1})(\d h),
\end{equation*}
we obtain from \eqref{eq:estimate_ball_alpha} that
\[\int_{\overline{B}_H(1/m)}\norm{h}^2\, (\lambda \circ \Phi^{-1})(\d h)=m^{\alpha-2}\frac{\alpha}{2-\alpha}(\lambda \circ \Phi^{-1})(\overline{B}_H^c)\leq m^{\alpha-2}\frac{\alpha}{2-\alpha} c_\alpha \norm{\Phi}_{\mathcal{L}_2(U, H)}^\alpha.\]

Applying similar arguments for the integrand in the second integral in  \eqref{estimate_alpha_norm} yields
\[\int_{\overline{B}_H(m)^c}{\abs{h}}\, (\nu\circ\Phi^{-1})(\d h)\leq m^{1-\alpha} \frac{\alpha}{\alpha-1}c_\alpha \abs{\Phi}_{\mathcal{L}_2(U, H)}^\alpha,\]
for all $\Phi \in \mathcal{L}_2(U, H)$, which completes the proof as in \eqref{estimate_alpha_norm} we can set
\begin{align}
    d_\alpha^m:=(m^{\alpha-2}+m^{1-\alpha})c_\alpha\to 0, \quad m\to\infty.
\end{align}
\end{proof}
If $L$ is an $\alpha$-stable cylindrical Lévy process and $T\colon H\to U$ a Hilbert-Schmidt operator, then
the cylindrical random variable $TL(1)$ is induced by a genuine stable random variable on $U$ with L\'evy measure $\lambda\circ T^{-1}$. This L\'evy measure depends continuously on $T$ in the following way:
\begin{Lemma}\label{le.conv_of_radonif_rv}
Let $\lambda$ be the cylindrical L\'evy measure of an $\alpha$-stable cylindrical Lévy process for $\alpha\in (1,2)$. Then for each $r>0$, the mapping
$\Phi\mapsto\lambda\circ\Phi^{-1}\arrowvert_{\overline{B}_H(r)^c}$ is continuous from $\mathcal{L}_2(U,H)$ to the space of Borel measures on $\overline{B}_H(r)^c$ equipped with the weak topology.
\end{Lemma}
\begin{proof}
If $\mu$ is a cylindrical probability measure and $(F_n)_{n \in \mathbb{N}}$ is a sequence converging to $F$ in $\mathcal{L}_2(U,H)$ then $(\mu\circ F_n^{-1})_{n \in \mathbb{N}}$ converges weakly to $\mu\circ F^{-1}$ according to \cite[Le.\ 2.1]{BR}.
From this, the assertion follows from \cite[Th.\ 5.5]{PA}.
\end{proof}

\subsection{Stochastic integration}

We briefly recall some facts on stochastic integration with respect to an $\alpha$-stable
cylindrical L\'evy process $L$ as introduced in \cite{BR}. A process $G\colon \Omega\times [0,T]\rightarrow\mathcal{L}_2(U,H)$ is called adapted and simple if it is of the form
\begin{align}
    G=\Phi_0\mathbbm{1}_{\lbrace 0\rbrace}+\sum_{i=1}^N\Phi_i\mathbbm{1}_{(t_{i-1}, t_i]},
\label{def:simple_process}
\end{align}
where $N\in\mathbb{N}$, $0=t_0<t_1<\dots<t_N=T$ and $\Phi_i$ is an $\mathcal{F}_{t_{i-1}}$-measurable and $\mathcal{L}_2(U,H)$-valued random variable taking finitely many values. We denote by $\mathcal{S}_{\rm adp}^{\rm HS}$ the class of all adapted, simple processes.
The integral process $\int_0^\cdot G\, \d L$ is defined as the sum of the Radonified increments
\begin{align}
    \int_0^tG{\rm d}L:=\sum_{i=1}^N\Phi_i\big(L({t_{i}\land t})-L({t_{i-1}\land t})\big), \quad t\in [0,T].
    \label{def:simple_integral}
\end{align}
Here, $\Phi_i\big(L({t_{i}\land t})-L({t_{i-1}\land t})\big)$ is defined as the $H$-valued random variable satisfying 
\begin{align*}
    \left\langle \Phi_i\big(L({t_{i}\land t})-L({t_{i-1}\land t})\big), h\right\rangle
    =\big(L({t_{i}\land t})-L({t_{i-1}\land t})\big)(\Phi_i^\ast h)
    \qquad\text{for all }h\in H.
\end{align*}

Let $\mathcal{S}_{\rm adp}^{1, \rm op}$ denote the class of adapted, simple $\mathcal{L}(H)$-valued processes bounded in the operator norm by $1$ on $[0,T]$.
An arbitrary predictable process $G\colon\Omega\times [0,T]\rightarrow\mathcal{L}_2(U,H)$ is stochastically integrable 
if there exists a sequence of adapted simple processes $(G_n)_{n\in \mathbb{N}}\subset\mathcal{S}_{\rm adp}^{\rm HS}$ such that:
\begin{itemize}
    \item[(i)] $(G_n)_{n \in \mathbb{N}}$ converges to $G$ $P\otimes{\rm Leb}\vert_{[0,T]}$-almost everywhere,
    \item[(ii)]$\displaystyle 
            \lim_{m,n \rightarrow \infty}\sup_{\Gamma \in \mathcal{S}_{{\rm adp}}^{1, {\rm op}}}E\Bigg[\norm{\int_0^T \Gamma(G_m-G_n) \;{\rm d}L}\wedge1 \Bigg]=0 $
\end{itemize}
In this case, $ \int_0^{\cdot}G \,{\rm d}L$ is defined as the limit of $\int_0^{\cdot}G_n\, {\rm d}L$ 
in the topology of uniform convergence in probability on $[0,T]$. 

It is shown in \cite{BR}, that a predictable process $G$ is stochastically integrable if and only it is an element of  $L_P^0\big(\Omega,L_{{\rm Leb}}^{\alpha}\big([0,T],\mathcal{L}_2(G,H)\big)\big)$.
%
It follows from  \cite[Co.\ 3]{KR} that for every $0<p<\alpha$ we have
\begin{align}\label{eq:moment_estimate_integral}
    \E\left[\sup_{t\in [0,T]}\abs{\int_0^tG\, {\rm d}L}^p\right]\leq e_{p,\alpha}\left( \E\left[\int_0^{T}\abs{G(t)}_{\mathcal{L}_2(U,H)}^\alpha \,  {\rm d }t\right]\right)^{p/\alpha},
\end{align}
for every stochastically integrable predictable process $G$, where $e_{p,\alpha}=\frac{\alpha}{\alpha-p}e_{2,\alpha}^{p/\alpha}$ for some $e_{2,\alpha}\in (0,\infty)$ that depends only on $\alpha$.
\begin{Lemma}\label{le.local_martingle_property}
If $G$ is a predictable stochastic process stochastically integrable with respect to the $\alpha$-stable cylindrical L\'evy process $L$ for some $\alpha \in (1,2)$ then $\int_0^\cdot G\, {\rm d}L $ is a local martingale.
\end{Lemma}

\begin{proof}
Define the predictable stopping times
$   \tau_n=\inf\left\{ t>0:\int_0^t\abs{G(s)}_{\mathcal{L}_2(U,H)}^\alpha {\rm d}s>n \right\}$ for $n\in\mathbb{N}$.
It follows from Proposition 4.22(ii) and Lemma 1.3 in \cite{DZ} that for each $n\in\N$ there exists a sequence of adapted, simple processes $( G_{n,k})_{k\in \mathbb{N}}$ such that
\begin{align}
 \lim_{k\to\infty} \E\left[\int_0^T\abs{G(s)\mathbbm{1}_{[0,\tau_n]}(s)-G_{n,k}(s)}_{\mathcal{L}_2(U,H)}^\alpha \,\d s\right]= 0.
\label{proof:local_martingale}
\end{align}
Since inequality \eqref{eq:moment_estimate_integral} guarantees for each $k, n\in\N$ that 
\begin{align*}
    \E\left[\sup_{0\leq t\leq T}\abs{\int_0^tG_{n,k}\, {\rm d}L}\right]\leq e_{1,\alpha}\left(\E\left[\int_0^T\abs{G_{n,k}(s)}_{\mathcal{L}_2(U,H)}^\alpha {\rm d}s\right]\right)^{1/\alpha\, }<\infty,
\end{align*}
the same arguments as in \cite[Th.\ I:51]{PR} show that the processes $\int_0^\cdot G_{n,k}\, {\rm d}L$
are martingales. Equation \eqref{proof:local_martingale} shows that $\int_0^\cdot G\mathbbm{1}_{[0,\tau_n]}\, {\rm d}L$
is a limit of martingales in $L^1(\Omega,H)$ by \eqref{eq:moment_estimate_integral}, and thus a martingale. Since standard arguments, e.g.\  \cite[Th.\ I.12]{PR}, establish
\begin{align}
    \left(\int_0^\cdot G\, {\rm d}L\right)^{\tau_n}=\int_0^\cdot G\mathbbm{1}_{[0,\tau_n]} \, {\rm d}L \quad \text{a.s.},
\label{eq:stopped_integral}
\end{align}
for the stopped integral process, the proof is completed.
\end{proof}

\begin{Thm}[Stochastic Fubini Theorem]\label{thm.fubini}
    Let $L$ be the standard symmetric $\alpha$-stable cylindrical Lévy process for  $\alpha\in (1,2)$. If $G\colon \Omega\times [0,T]^2\rightarrow\mathcal{L}_2(U,H)$ is measurable, $G(t,\cdot)$ is predictable for every $t\in [0,T], $ and
    $\int_0^T\int_0^T\norm{G(t,s)}_{\mathcal{L}_2(U,H)}^\alpha\, {\rm d}t \, {\rm d}s<\infty$ a.s.\ then it follows: 
    \begin{itemize}
        \item[{\rm (a)}] $G(t,\cdot)$ is stochastically integrable for every $t\in [0,T]$ and $\int_0^T G(\cdot,s){\rm d}L(s)$ is a.s.\ Bochner integrable;
        \item[{\rm (b)}] $G(\cdot,s)$ is a.s.\ Bochner integrable for every $s\in [0,T]$ and $\int_0^T G(t,\cdot)\,{\rm d}t$ is stochastically integrable;
        \item[{\rm (c)}] $\displaystyle \int_0^T\left(\int_0^TG(t,s)\,{\rm d}t\right)\,{\rm d}L(s)=\int_0^T\left(\int_0^TG(t,s)\,{\rm d}L(s)\right)\, {\rm d}t$ a.s.
    \end{itemize}
\end{Thm}
\begin{proof} The proof is similar as in finite dimensions; see \cite{ZX}. 
\end{proof}

\subsection{Random measures and compensators}

In this section, we briefly recall some results on random measures and their compensators from \cite[Ch.\ II]{JS}.

\begin{Definition}[Random measure]
A family $\mu=\lbrace\mu(\omega): \omega\in\Omega\rbrace$ is called a random measure on $[0,T]\times H$ if $\mu(\omega)$ is a measure on $\mathcal{B}([0,T])\otimes\mathcal{B}(H)$ for each $\omega\in\Omega$. It is said to be an integer-valued random measure if moreover, we have
\begin{itemize}
    \item[{\rm (i)}] $\mu (\lbrace t\rbrace\times H)\leq 1$ for all $t\in [0,T]$ $P$-a.s.;
    \item[{\rm (ii)}] $\mu$ takes values in $\mathbb{N}\cup\lbrace\infty\rbrace$ $P$-a.s.
\end{itemize}
\end{Definition}

We denote by $\Tilde{\mathcal{P}}$ (resp.\ $\Tilde{\mathcal{O}}$) the \textit{predictable} (resp.\ \textit{optional}) $\sigma$-algebra on $\Omega\times  [0,T]\times H$ and call a function $W:\Omega\times [0,T]\times H \mapsto\mathbb{R}$ \textit{predictable} (resp.\ \textit{optional}) if it is $\Tilde{\mathcal{P}}$ (resp.\ $\Tilde{\mathcal{O}}$) measurable.

If $\mu$ is a random measure and $W$ is optional we define
\begin{align*}
&    \left(\int_0^t\int_HW(s,h)\,\mu({\rm d}s, {\rm d}h)\right)(\omega)\\
    &\qquad\qquad :=\begin{cases}\int_0^t\int_HW(\omega,s,h)\,\mu(\omega)({\rm d}s,{\rm d}h),& \text{if} \int_0^t\int_H|W(\omega,s,h)|\,\mu(\omega)({\rm d}s,{\rm d}h)<\infty,\\
    \infty,& \text{otherwise.}
    \end{cases}
\end{align*}
A random measure $\mu$ is called \textit{predictable} (resp.\ \textit{optional}) if 
$(\int_0^t\int_HW(s,h)\, \mu({\rm d}s,{\rm d}h):\, t\in [0,T])$
is predictable (resp.\ optional) for every predictable (resp.\ optional) function $W$. An optional random measure $\mu$ is called $\sigma$-finite if there exists a sequence $(A_n )_{n\in \mathbb{N}}\subset\Tilde{\mathcal{P}}$ with  $\bigcup_{n=1}^\infty A_n=\Omega\times [0,T]\times H$, such that
$\E\left[\int_0^T\int_{H}\mathbbm{1}_{A_n}(s,h)\,\mu({\rm d}s,{\rm d}h)\right]<\infty $ for each $n\in\mathbb{N}$.

For each  $\sigma$-finite, optional measure $\mu$ on $[0,T]\times H$ there exists a predictable random measure $\nu$ on $\mathcal{B}([0,T])\otimes \mathcal{B}(H)$ such that
\begin{align}
\E\left[\int_0^t\int_HW(s,h)\,\mu({\rm d}s,{\rm d}h)\right]=\E\left[\int_0^t\int_HW(s,h)\,\nu({\rm d}s,{\rm d}h)\right]
\label{eq:compensator_expectation}
\end{align}
for all $t\in [0,T]$, and any non-negative predictable function $W$. The measure $\nu$ is determined uniquely up to a set of probability zero by \eqref{eq:compensator_expectation} and is called the \textit{compensator} of $\mu$; see \cite[th.\ II.1.8]{JS}.

 If $Y$ is an $H$-valued, adapted c\`adl\`ag process then the integer-valued random measure $\mu^Y$ characterised by
\begin{align*}
    \mu^Y((0,t]\times B)=\sum_{0\leq s\leq t}\mathbbm{1}_B(\Delta Y(s)), \quad t\in (0,T], B\in\mathcal{B}(H), 0\notin B,
\end{align*}
where $\Delta Y(s):=Y(s)-\lim_{h\searrow 0+}Y(s-h)$ for $s\in [0,T]$, is an optional and $\sigma$-finite random measure on $\mathcal{B}([0,T])\otimes \mathcal{B}(H)$. Thus, its compensator exists which we denote by $\nu^Y$.

%

\begin{Example}
Let $L$ be a genuine $H$-valued L\'evy process with L\'evy measure $\lambda$. Then the compensator $\nu^L$ of the jump measure $\mu^L$ is given as the extension of $\mu^L((s,t]\times B)=(t-s)\lambda (B)$, $0\leq s<t\leq T$, $B\in\mathcal{B}(H)$ to $\mathcal{B}([0,T])\otimes\mathcal{B}(H)$. 
\label{example:compensator_levy}
\end{Example}

In the sequel, we will make use of another characterisation of compensators of jump-measures. We denote by $\mathcal{C}^+(H)$ the class of non-negative, continuous functions $k:H \rightarrow \mathbb{R}$ bounded on $H$ and vanishing inside a neighbourhood of $0$.
\begin{Proposition}\label{equ_char_of_comp}
The compensator $\nu^Y$ of the jump-measure $\mu^Y$ of  an $H$-valued c\`adl\`ag semimartingale $Y$
is characterised by being predictable and satisfying either of the following:
\begin{itemize}
    \item[{\rm (i)}] The process
    \begin{align*}
       \left( \int_0^t\int_Hk(h)\,\mu^Y ({\rm d}s,{\rm d}h)-\int_0^t\int_Hk(h)\,\nu^Y({\rm d}s,{\rm d}h):\; t\in [0,T]\right)
    \end{align*}
    is a local martingale for every $k\in\mathcal{C}^+(H)$.
    \item[{\rm (ii)}] If $W$ is predictable and the process
    \begin{align}
        \left(\int_0^t\int_HW(s,h)\,\mu^Y({\rm d}s,{\rm d}h):\;t\in [0,T]\right)
    \label{eq:compensator_local_martingale}
    \end{align}
    is locally integrable, then so is
    \begin{align*}
      \left(\int_0^t\int_HW(s,h)\,\nu^Y({\rm d}s,{\rm d}h):\; t\in [0,T]\right)
    \end{align*}
    and
    \begin{align*}
     \left(\int_0^t\int_HW(s,h)\,\mu^Y({\rm d}s,{\rm d}h)-\int_0^t\int_HW(s,h)\,\nu^Y({\rm d}s,{\rm d}h):\; t\in [0,T]\right)
    \end{align*}
    is a local martingale.
\end{itemize}
\end{Proposition}
\begin{proof}
The equivalence between (i) and (ii) follows by the same argument as in the proof of \cite[Th.\ II.2.21.]{JS}. The fact that (ii) is an equivalent definition of the compensator is proved in \cite[Th.\ II.1.8.]{JS}.
\end{proof}

Proposition \ref{equ_char_of_comp} justifies the following standard notation: if $W$ is predictable and (\ref{eq:compensator_local_martingale}) is locally integrable, we define the following local martingale
\begin{align*}
    \int_0^t\int_HW(s,h)\,(\mu^Y-\nu^Y)({\rm d}s,{\rm d}h):=\int_0^t\int_HW(s,h)\,\mu^Y({\rm d}s,{\rm d}h)-\int_0^t\int_HW(s,h)\,\nu^Y({\rm d}s,{\rm d}h)
\end{align*}
for each $t\in [0,T$].

\section{Predictable compensator}

For an $\alpha$-stable cylindrical L\'evy process $L$ for some $\alpha\in (1,2)$ and a stochastically integrable predictable process $G$, we define the integral process $X=\int_0^\cdot G{\rm d}L$ and 
\begin{align}
    \nu\left((0,t]\times B\right):=\int_0^t\left(\lambda\circ G(s)^{-1}\right)(B)\,{\rm d}s
    \qquad \text{for each $t\in (0,T], B\in\mathcal{B}(H)$ with $0 \notin \Bar{B}$.}
        \label{def:measure_nu}
\end{align}
The main result of this section is that $\nu$ extends to a random measure on $\mathcal{B}([0,T])\otimes \mathcal{B}(H)$ and that the extension is the predictable compensator of the jump measure of $X$. We will derive this result by a couple of Lemmata. 
\begin{Lemma}
The set function $\nu$ defined in \eqref{def:measure_nu}
is well defined and extends to a predictable random measure on $\mathcal{B}([0,T])\otimes \mathcal{B}(H)$. This extension is unique among the class of $\sigma$-finite random measures on $\mathcal{B}([0,T])\otimes \mathcal{B}(H)$ that assign $0$ mass to the origin.
\label{lemma:compensator_is_predictable}
\end{Lemma}
\begin{proof}
\textit{Step 1:} We show that for all open sets $B \subseteq H$ with $0 \notin \Bar{B}$ the process
\begin{align*}
    f\colon\Omega\times [0,T] \rightarrow \mathbb{R}, \quad f(\omega, t)=\left(\lambda\circ G(\omega,t)^{-1}\right)(B)
\end{align*}
is predictable. Since the function $h\colon \mathcal{L}_2(U,H) \rightarrow \mathbb{R}$ defined by $h(\Phi)= (\lambda \circ \Phi^{-1})(B)$ is lower semicontinuous by Lemma \ref{le.conv_of_radonif_rv} and the Portmonteau Theorem as the set $B$ assumed to be open,  $h$ is measurable. Since $G\colon\Omega\times [0,T] \rightarrow \mathcal{L}_2(U,H)$ is predictable, it follows that $f=h\circ G$ is predictable. \\
\textit{Step 2:} We show that $f$ is predictable for all $B \in \mathcal{B}(H \setminus \{0\})$, which will immediately imply that \eqref{def:measure_nu} is almost surely well defined and predictable as it is then just an integral of a non-negative predictable process. We define
\[\mathcal{D}=\left\{B \in \mathcal{B}(H \setminus \{0\}): \lambda \circ G(\cdot,\cdot)^{-1}(B)\, \text{ is predictable}\right\}, \]
and claim that $\mathcal{D}$ is a $\lambda$-system. Continuity of measures implies that $H \setminus \{0\} \in \mathcal{D}$ since, for all $t\in (0,T]$ and $\omega\in \Omega$, we have 
\[\left(\lambda\circ G(\omega,t)^{-1}\right)(H\setminus \{0\})=\lim_{n \rightarrow \infty}\left(\lambda\circ G(\omega,t)^{-1}\right)\left(\Bar{B}_H\left(1/n\right)^c\right),\]
where the right hand side is the limit of processes that are predictable by Step 1.  If $B \in \mathcal{D}$ then $B^c \in \mathcal{D}$ since
 \[\left(\lambda\circ G(\omega,t)^{-1}\right)(B^c)=\left(\lambda\circ G(\omega,t)^{-1}\right)(H\setminus \{0\})-\left(\lambda\circ G(\omega,t)^{-1}\right)(B). \]
The collection $\mathcal{D}$ is closed under union of increasing sequences, which follows as above from continuity of measures and predictability of the pointwise limit.  This concludes the proof of the claim that $\mathcal{D}$ is a $\lambda$-system.

We define the $\pi$-system. 
\[\mathcal{I}=\left\{B \in \mathcal{B}(H \setminus \{0\}): B\, \text{ is open}\right\}.\]
The family $\mathcal I$ is contained in $\mathcal D$, since for each $B\in \mathcal D$ we have
\[\left(\lambda \circ G(\omega,t)^{-1}\right)(B)=\lim_{n \rightarrow \infty}\left(\lambda \circ G(\omega,t)^{-1}\right)\left(B \cap \Bar{B}_H(1/n)^c\right),\]
and the right-hand side is predictable by Step 1. The Dynkin $\pi$-$\lambda$ theorem for sets, see e.g.\ \cite[Th.\ 1.1]{OK} implies  $\sigma(\mathcal{I})\subseteq \mathcal{D}$, and thus $\mathcal{D}=\mathcal{B}(H \setminus \{0\})$.
\\
\textit{Step 3:}
Let $\omega\in \Omega$ be such that $   \int_0^T\norm{G(\omega, s)}_{\mathcal{L}_2(U,H)}^\alpha{\rm d}s<\infty$.  
Equation \eqref{def:measure_nu} defines  the set function $\nu(\omega)$ on the semi-ring 
\[\mathcal{S}=\left\{(0,t]\times B: t\in [0,T] \text{ and } B\in\mathcal{B}(H \setminus \{0\})\right\}.\]
The set function $\nu(\omega)$ is $\sigma$-additive by its very definition and $\sigma$-finite, since for $n\in{\mathbb N}$ we have by \eqref{eq:estimate_ball_alpha} and \eqref{eq:stable_ball_radius} that 
\begin{align*}
    \nu(\omega)\left((0,T]\times \overline{B}_H^c\left(1/n\right)\right)=& \int_0^T\left(\lambda \circ G(\omega,s)^{-1}\right)\left(\overline{B}_H^c\left(1/n\right)\right)\,{\rm d}s\\
    =&\,n^\alpha \int_0^T\left(\lambda \circ G(\omega,s)^{-1}\right)\left(\overline{B}_H^c\right)\,{\rm d}s\\
    \leq &\,n^\alpha\, c_\alpha \int_0^T\norm{G(\omega, s)}_{\mathcal{L}_2(U,H)}^\alpha{\rm d}s<\infty.
\end{align*}
Carathéodory's extension theorem, see e.g.\ \cite[Th.\ 2.5]{OK}, implies that the set function $\nu(\omega)$ extends uniquely to a measure on $\mathcal{B}([0,T])\otimes \mathcal{B}(H\setminus\lbrace 0\rbrace)$ which we also denote by $\nu(\omega)$. 

\textit{Step 4:} It remains to show that $\nu$ is predictable.
Applying the monotone class theorem as above shows that the process $\int_0^\cdot\left(\lambda\circ G(s)^{-1}\right)(B)\,{\rm d}s$ is predictable for each $B\in \mathcal{B}(H \setminus \{0\})$.
Since
\begin{align*}
\int_0^\cdot\mathbbm{1}_{(s,t]}(u)\mathbbm{1}_A\left(\lambda\circ G(u)^{-1}\right)(B){\rm d}u
    =\int_0^\cdot\left(\lambda\circ \left(\mathbbm{1}_{(s,t]}(u)\mathbbm{1}_AG(u)\right)^{-1}\right)(B){\rm d}u,
\end{align*}
it follows that the process   $\int_0^\cdot\int_HW(u,h)\Tilde{\nu}({\rm d}u,{\rm d}h)(\cdot)$ is predictable 
for all functions $W=\mathbbm{1}_{(s,t]}\mathbbm{1}_A\mathbbm{1}_B$ with $0<s<t\leq T$, $A\in\mathcal{F}_s$ and  $B\in\mathcal{B}(H\setminus \lbrace 0\rbrace)$. An application of the  functional monotone class theorem (follows e.g. from \cite[Th.\ 3.14]{williams1991probability})
extends this result to all predictable processes $W$ on $\Omega\times [0,T]\times H$, which shows 
predictability of the random measure $\nu$ on $\mathcal{B}([0,T])\otimes \mathcal{B}(H\setminus \lbrace 0\rbrace)$.
Defining $\nu((s,t]\times\{0\}):=0 $ for any $0\leq s<t\leq T$ extends $\nu$ to a predictable random measure on $\mathcal{B}([0,T])\otimes \mathcal{B}(H)$.

\end{proof}

To  show that the random measure $\nu$ characterised by \eqref{def:measure_nu} is the compensator of the jump-measure $\mu^X$ of the integral process $X$, we  first consider the case when the integrand is an adapted, simple process.
\begin{Lemma}\label{step_compensator}
Suppose that $G$ is an adapted, simple process in $\mathcal{S}_{\rm adp}^{\rm HS}$. Then the random measure $\nu$ 
obtained in Lemma \ref{lemma:compensator_is_predictable} is the predictable compensator of $\mu^X$.
\end{Lemma}
\begin{proof}
Since Lemma \ref{lemma:compensator_is_predictable} guarantees that $\nu$ is predictable, it remains to show \eqref{eq:compensator_expectation}, which by the functional monotone class theorem reduces to proving 
\begin{align*}
    \E\left[\mathbbm{1}_A\sum_{s<u\leq t}\mathbbm{1}_B(\Delta X(u))\right]=\E\left[\mathbbm{1}_A\int_s^t\left(\lambda\circ G(u)^{-1}\right)(B)\,\d u\right]
\end{align*}
for any $0<s<t\leq T$, $A\in\mathcal{F}_s$ and $B\in\mathcal{B}(H)$ with  $0\not\in\Bar{B}$. 
Let $G$ be of the form \eqref{def:simple_process}, and assume that the points of the partition contain $s$ and $t$; otherwise these can be added. Then $X$ takes the form \eqref{def:simple_integral}, and it follows  
\begin{align*}
    \mathbbm{1}_A\sum_{s<u\leq t}\mathbbm{1}_B(\Delta X(u))
    &=\mathbbm{1}_A\sum_{i=1}^N\sum_{s \leq t_{i-1}<u\leq t_i \leq t}\mathbbm{1}_B(\Delta G(t_i) L(u))\\
    &=\mathbbm{1}_A\sum_{i=1}^N\sum_{s \leq t_{i-1}<u\leq t_i \leq t}\mathbbm{1}_B(\Delta\Phi_i L(u)).
\end{align*}
For each $i\in \{1,\dots, N\}$, the random variable $\Phi_i$ is of the form
$\Phi_i=\sum_{j=1}^{m_i}\mathbbm{1}_{A_{i,j}}\phi_{i,j}$ for some pairwise disjoint sets $A_{i,j}\in\mathcal{F}_{t_{i-1}}$ and  $\phi_{i,j}\in\mathcal{L}_2(U,H)$ for $j\in\lbrace 1,\dots,m_i\rbrace$.
Since $0\not\in\Bar{B}$, we have
\begin{align*}
    \E\left[\mathbbm{1}_A\sum_{t_{i-1}<u\leq t_i}\mathbbm{1}_B(\Delta\Phi_i L(u))\right]&=\sum_{j=1}^{m_i}\E\left[\mathbbm{1}_{A\cap A_{i,j}}\sum_{t_{i-1}<u\leq t_i}\mathbbm{1}_B(\Delta\phi_{i,j}L(u))\right]\\
    &=\sum_{j=1}^{m_i}(t_i-t_{i-1})\E\left[\mathbbm{1}_{A\cap A_{i,j}}\left(\lambda\circ\phi_{i,j}^{-1}\right)(B)\right]\\
    &=(t_i-t_{i-1})\E\left[\mathbbm{1}_{A}\left(\lambda\circ\Phi^{-1}_i\right)(B)\right],
\end{align*}
because $A\cap A_{i,j}\in\mathcal{F}_{t_{i-1}}$ and the compensator of the jump measure of the L\'evy process $\phi_{i,j}L$ in $H$ is given by $\left(\lambda\circ\phi_{i,j}^{-1}\right)\,\d h\,{\rm d}t$ since its L\'evy measure is $\left(\lambda\circ\phi_{i,j}^{-1}\right)$, see Example \ref{example:compensator_levy}.
\end{proof}

Before we show that the result of Lemma \ref{step_compensator} can be extended to general integrands, we need to prove some technical Lemmata. Recall the class of functions $\mathcal{C}^+(H)$ used in Proposition \ref{equ_char_of_comp} (and defined just before) to determine the compensator.
\begin{Lemma}\label{limit_of_jump_measure}
Let $( f_n)_{n\in \mathbb{N}}$ be a sequence of  c\`adl\`ag functions $f_n\colon[0,T]\to H$ converging uniformly to $f\colon[0,T]\rightarrow H$. Then we have for any $k \in \mathcal{C}^+(H)$ that 
\begin{align}
    \lim_{n \rightarrow \infty}\sup_{t \in [0,T]} \left \vert\sum_{0\leq s\leq t}k(\Delta f_n(s))-\sum_{0\leq s\leq t}k(\Delta f(s))\right \vert=0.
\label{eq:limit_of_jump_measure}
\end{align}
\end{Lemma}
\begin{proof}
Both sums in (\ref{eq:limit_of_jump_measure}) are finite by the c\`adl\`ag property of $f,f_n$ and since $k$ vanishes inside a neighbourhood of $0$. The assumed uniform convergence implies 
\begin{align}
  \lim_{n\to\infty}  \sup_{t\in [0,T]}\left\Vert\Delta f_n(t)-\Delta f(t)\right\Vert = 0.
    \label{proof:limit_of_jump_measure_2}
\end{align}
Denoting  $\text{supp}(k):=\lbrace h\in H: k(h)\neq 0\rbrace$ and $\delta:=\frac{1}{2}\text{dist}(0,\text{supp}(k))$, 
we obtain that $
    \text{supp}(k)_\delta:=\lbrace h\in H,\text{dist}(h,\text{supp}(k))<\delta\rbrace,
$ is bounded away from zero, i.e. $  0\notin\overline{\text{supp}(k)_{\delta}}\label{eq.0_not_in_support}$.
It follows that the set $ D:=\lbrace t\in [0,T]: \Delta f(t)\in\text{supp}(k)_\delta\rbrace $
is finite, which together with continuity of $k$ and \eqref{proof:limit_of_jump_measure_2} implies
\begin{align}
\lim_{n \rightarrow \infty}\sup_{t\in D}{|k(\Delta f_n(t))-k(\Delta f(t))|}= 0.
\label{proof:limit_of_jump_measure}
\end{align}
Since \eqref{proof:limit_of_jump_measure_2} guarantees that there exists $n_0\in\mathbb{N}$ such that we have $ \Delta f_n(t)\notin\text{supp}(k)$ for all $n\ge n_0$ and $t\in [0,T]\setminus D $,
we conclude from  \eqref{proof:limit_of_jump_measure} for $n \geq n_0$ that
\begin{align*}
    \sup_{t\in [0,T]}\left\vert\sum_{0\leq s\leq t}k(\Delta f_n(s))-\sum_{0\leq s\leq t}k(\Delta f(s))\right\vert=&\sup_{t\in [0,T]}\left\vert\sum_{s\in D\cap [0,t]}k(\Delta f_n(s))-\sum_{s\in D\cap [0,t]}k(\Delta f(s))\right\vert\\
    \leq & |D|\sup_{t\in D}{\left \vert k(\Delta f_n(t))-k(\Delta f(t))\right \vert}\to 0, \quad n\to\infty.
\end{align*}
The proof is complete.
\end{proof}

\begin{Lemma}\label{limit_of_compensators}
Let $g_n, g\in L_{\rm Leb}^\alpha([0,T],\mathcal{L}_2(U,H))$, $n\in\N$, be such that $g_n$ converges to $g$ in 
$L_{\rm Leb}^\alpha([0,T],\mathcal{L}_2(U,H))$ and pointwise for almost every $s\in [0,T]$.  Then we obtain for each  $k \in \mathcal{C}^+(H)$ that
\begin{align*}
    \lim_{n \rightarrow \infty}\sup_{t \in [0,T]}\left \vert\int_0^t\int_H k(h)\, \left(\lambda \circ g_n(s)^{-1}\right)\,{\rm d}h\,{\rm d}s-\int_0^t\int_H k(h)\,\left(\lambda \circ g(s)^{-1}\right)\,{\rm d}h\,{\rm d}s\right \vert=0.
\end{align*}
\end{Lemma}
\begin{proof}
Lemma \ref{le.conv_of_radonif_rv} implies for almost each $s \in [0,T]$ and every $n\in\N$ that
\begin{align*}
    \lim_{n \rightarrow \infty}\int_H k(h)\, \left( \lambda \circ g_n(s)^{-1} \right)\,{\rm d}h = \int_H k(h)\, \left( \lambda \circ g(s)^{-1} \right)\,{\rm d}h.
\end{align*}
Since $k$ is bounded and vanishes in a neighbourhood of $0$, we conclude from inequality \eqref{eq:estimate_ball_alpha}
\begin{align*}
    \int_H k(h)\, \left( \lambda \circ g_n(s)^{-1} \right)\,{\rm d}h \leq c_{k,\alpha}\abs{g_n(s)}_{\mathcal{L}_2(U,H)}^{\alpha},
\end{align*}
for a constant $c_{k,\alpha}$ independent of $s\in[0,T]$ and $n\in\mathbb{N}$. Since for each $t\in [0,T]$ we have
\begin{align*}
    \lim_{n \rightarrow \infty} \int_0^t  \abs{g_n(s)}_{\mathcal{L}_2(U,H)}^{\alpha}\,{\rm d}s= \int_0^t\abs{g(s)}_{\mathcal{L}_2(U,H)}^{\alpha}\,{\rm d}s,
\end{align*}
 the generalised Lebesgue's dominated convergence theorem, see e.g.\ \cite[Th.\ 4.19]{RF}, implies \begin{align*}
    \lim_{n \rightarrow \infty}\int_0^t\int_H k(h)\, \left( \lambda \circ g_n(s)^{-1} \right)\,{\rm d}h\,{\rm d}s = \int_0^t\int_H k(h)\, \left( \lambda \circ g(s)^{-1} \right)\,{\rm d}h\,{\rm d}s.
\end{align*}
As the functions 
\begin{align*}
   t\mapsto \int_0^t\int_H k(h)\, \left( \lambda \circ g_n(s)^{-1} \right)\,{\rm d}h\,{\rm d}s
\end{align*}
are continuous monotone and converge pointwise to a continuous limit on $[0,T]$, the convergence is uniform 
by \cite[p.\ 81/127]{PSZ} (or deuxième théorème de Dini). 
\end{proof}

Now we can prove the main result of this section.

\begin{Thm}\label{th.compensator}
Let $L$ be an $\alpha$-stable cylindrical L\'evy process $L$ for some $\alpha\in (1,2)$ and $G$ a stochastically integrable predictable process. Then the predictable compensator $\nu^X$ of the jump measure $\mu^X$ of  $X:=\int_0^\cdot G{\rm d}L$ is characterised by \eqref{def:measure_nu}.
\end{Thm}
\begin{proof}
In light of Proposition \ref{equ_char_of_comp}, it suffices to show that the process $M^k$ defined by
\begin{align*}
    \left(M^k(t):= \int_0^t\int_H k(h)\, \mu^{X}({\rm d}s, {\rm d}h) - \int_0^t\int_H k(h)\, \left(\lambda \circ G(s)^{-1}\right)\,{\rm d}h\,{\rm d}s, \, t\in [0,T]\right),
\end{align*}
is a local martingale for any $k\in\mathcal{C}^+(H)$. Lemma 4.3 in  \cite{BR} guarantees that there exists a sequence $( G_n)_{n\in \mathbb{N}}$ of adapted, simple processes in $\mathcal{S}_{\rm adp}^{\rm HS}$ converging both in $L_{\rm Leb}^\alpha([0,T],L_2(U,H))$ \text{a.s.} and $P\otimes{\rm Leb}\vert_{[0,T]}-\rm a.e.$\ to $G$. Letting 
$X_n:=\int_0^\cdot G_n{\rm d}L$ and denoting  the jump-measure of $X_n$ by $\mu^{X_n}$, we define 
for each $k \in \mathcal{C}^+(H)$ and $n \in \mathbb{N}$ a process $M_n^k$ by 
\begin{align*}
    \left(M_n^k(t):=\int_0^t\int_H k(h)\, \mu^{X_n}({\rm d}s, {\rm d}h)-\int_0^t\int_H k(h)\, \left(\lambda \circ G_n(s)^{-1}\right)\,{\rm d}h\,{\rm d}s,\,t\in [0,T]\right). 
\end{align*}
Proposition \ref{equ_char_of_comp} and Lemma \ref{step_compensator} imply that $M_n^k$ is a local martingale for all $n \in \mathbb{N}$. Since for each $n \in \mathbb{N}$ and $t\in [0,T]$ we have that $\mu^{X_n} (\lbrace t\rbrace\times H)\leq 1$ almost surely, it follows that
\begin{align*}
    \norm{\Delta\left( \int_0^t\int_H k(h)\, \mu^{X_n}({\rm d}h,{\rm d}s) \right)}\leq \norm{k}_{\infty} \quad\text{a.s.}, 
\end{align*}
which shows $\norm{\Delta M_n^k(t)}\leq \norm{k}_{\infty}$ a.s.\ for all $n \in \mathbb{N}$.

Almost sure uniform convergence of $X_n$ and Lemma \ref{limit_of_jump_measure} guarantee that there exists an $\Omega_1 \subseteq \Omega$ with $P(\Omega_1)=1$ such that, for all  $\omega \in \Omega_1$, we have
\begin{align}
    \lim_{n \rightarrow \infty}\sup_{t \in [0,T]}\left \vert \int_0^t\int_H k(h)\, \mu^{X_n(\omega)}({\rm d}h,{\rm d}s) - \int_0^t\int_H k(h)\, \mu^{X(\omega)}({\rm d}h,{\rm d}s)\right \vert = 0.
\label{proof:compensator1}
\end{align}
In the same way, by convergence of $G_n$ both in $L_{\rm Leb}^\alpha([0,T],L_2(U,H))$ \text{a.s.} and $P\otimes{\rm Leb}\vert_{[0,T]}-\rm a.e.$\ and Lemma \ref{limit_of_compensators} there exists an $\Omega_2 \subseteq \Omega$ with $P(\Omega_2)=1$ such that, for all $\omega \in \Omega_2$, we have
\begin{align}
    \lim_{n \rightarrow \infty}\sup_{t \in [0,T]}\left \vert\int_0^t\int_H k(h)\, \left(\lambda \circ G_n(\omega,s)^{-1}\right)\,{\rm d}h\,{\rm d}s-\int_0^t\int_H k(h)\, \left(\lambda \circ G(\omega,s)^{-1}\right)\,{\rm d}h\,{\rm d}s\right \vert=0.
\label{proof:compensator2}
\end{align}
Equations \eqref{proof:compensator1} and \eqref{proof:compensator2} show  that $M_n^k$ converges uniformly to $M^k$ almost surely. As the jumps of $M^k_n$ are a.s.\ uniformly bounded by $\norm{k}_{\infty}$, we conclude from \cite[Co.\ IX.1.19]{JS} that $M^k$ is a local martingale and the proof is complete.
\end{proof}

\section{Quadratic variation of the integral process}

The quadratic covariation of two real-valued c\`adl\`ag semimartingales $V_1$ and $V_2$ starting from zero is the process $[V_1,V_2]$ defined by
\begin{align*}
    \left[V_1, V_2\right](t):=V_1(t)V_2(t)-\int_0^tV_1(s-)\,{\rm d}V_2(s)-\int_0^tV_2(s-)\,{\rm d}V_1(s), \quad t\in [0,T].
\end{align*}
When $V:=V_1=V_2$, we call the process $\left[V\right]:=\left[V, V\right]$ the quadratic variation of $V$. The continuous part of $\left[V\right]$ is defined by
\begin{align}
    \left[V\right]^c(t)=\left[V\right](t)-\sum_{0\leq s\leq t}\left(\Delta V(s)\right)^2 \quad \text{ for each $t\in [0,T]$.}
    \label{def:quadratic_variation_purely_discontinuous}
\end{align}
If $\left[V\right]^c=0$ we say that $V$ is purely discontinuous; see e.g.\ \cite[Se.\ II.6]{PR}.

The concept of quadratic variation is generalised for a c\`adl\`ag semimartingale $Z$ with values in the separable Hilbert space $H$ 
in  \cite[Se.\ 26]{ME}. Let $(f_i)_{i\in\N}$ denote an orthonormal basis of $H$. There exists a 
unique stochastic process $\left[\left[Z\right]\right]$ with values in the Hilbert-Schmidt tensor product of $H$ satisfying
\begin{align*}
    \langle\left[\left[Z\right]\right], f_i\otimes f_j\rangle=\left[Z_i, Z_j\right] \quad\text{for all } i,j\in\mathbb{N},
\end{align*}
where $\otimes$ denotes the tensor product and $Z_i(t)=\langle Z(t),f_i\rangle $ for $t\in [0,T]$  
are the projection processes of $Z$; see \cite[Se.\ 21.2]{ME} for brief introduction. The process $\left[\left[Z\right]\right]$ does not depend on the choice of the orthonormal basis $(f_i)_{i \in \mathbb{N}}$. The process $\left[\left[Z\right]\right]$ is called the tensor quadratic variation of $Z$ and its continuous part  $\left[\left[Z\right]\right]^c$ is defined by
\begin{align*}
    \langle\left[\left[Z\right]\right]^c(t), f_i\otimes f_j\rangle=\langle\left[\left[Z\right]\right](t), f_i\otimes f_j\rangle-\sum_{0\leq s\leq t}\Delta \left(Z^i(s)Z^{j}(s)\right) \quad \text{for all }t\in [0,T], \, i,j\in\mathbb{N}.
\end{align*} 
We say that $Z$ is purely discontinuous if $\left[\left[Z\right]\right]^c$=0.

\begin{Proposition}
    Let $L$ be an $\alpha$-stable cylindrical L\'evy process for some $\alpha\in (1,2)$ and $G$ a stochastically integrable predictable process with values in ${\mathcal L}_2(U,H)$. Then the integral process $X:=\int_0^{\cdot} G\, \d L$ is purely discontinuous.
\label{prop:purely_discontinuous}
\end{Proposition}
\begin{proof}
We proceed in three steps. \\
\textit{Step 1:} Assume $H=\mathbb{R}$ and $U=\mathbb{R}^d$ for some $d\in\mathbb{N}$. In this case, $L$ is a $U$-valued standard symmetric $\alpha$-stable L\'evy process, and therefore purely discontinuous; see e.g. \cite[p.\ 71]{PR}. Pure discontinuity is preserved also for the integral process; see e.g.\ \cite[Se.\ IX.5.5a]{JS} or \cite[Th.\ II.29]{PR}. \\
\textit{Step 2:} Assume $H=\mathbb{R}$, but without any further restrictions on $U$. In that case, by the identification 
$U\simeq\mathcal{L}_2(U,\mathbb{R})$, the integrand $G$ is a $U$-valued process satisfying 
\begin{align}
    \int_0^T\abs{G(t)}^\alpha {\rm d}t<\infty \quad \text{a.s.}
\label{proof:purely_dicsontinuous}
\end{align}
 Fix an orthonormal basis $(f_k)_{k\in \mathbb{N}}$ in $U$ and define for each $n \in \mathbb{N}$ the projection
\begin{align*}
\pi_n\colon U\to U, \qquad \pi_n(u)=\sum_{k=1}^n\langle u, f_k\rangle f_k.
\end{align*}
Since the projection $\pi_n$ is a Hilbert-Schmidt operator, there exists a $U$-valued L\'evy process $L_n$ with the property
$\langle L_n, u\rangle=L(\pi_n^*u)$ for all $u\in U$.  We define the approximations
\begin{align*}
\quad X_n:=\int_0^\cdot G{\rm d}L_n, \quad n\in\mathbb{N}.
\end{align*}
Since $L_n$ attains values in a finite-dimensional subspace and is a symmetric $\alpha$-stable process by  \cite[Le.\ 2.4]{RI}, it follows that $X_n$ is purely discontinuous by \textit{Step 1}.

Let $M$ be a real-valued, continuous  martingale and define for $k\in\mathbb{N}$ the stopping times
\begin{align*}
    \tau_k=\inf\left\{ t>0:\int_0^t\abs{G(s)}^\alpha {\rm d}s\geq k\right\}\land\inf\left\{ t>0:|M(t)|\geq k\right\}\land T.
\end{align*}
It follows that $\tau_k\to T$ as $k\to \infty$ by \eqref{proof:purely_dicsontinuous}. Since $X_n$ is purely discontinuous, it follows from \cite[Le.\ I.4.14]{JS} that $(X_nM)^{\tau_k}$ is a  local martingale for each $k$, $n\in\mathbb{N}$.  Since applying inequality \eqref{eq:moment_estimate_integral} and equality \eqref{eq:stopped_integral} shows
\begin{align*}
    \E\left[\sup_{0\leq t\leq T}{|(X_nM)^{\tau_k}(t)|}\right]&=\E\left[\sup_{0\leq t\leq T}\left \vert M(t)^{\tau_k}\right \vert\left \vert\int_0^t\mathbbm{1}_{[0,\tau_k]}G\, {\rm d}L_n\right \vert\right]\\
    &\leq k\E\left[\int_0^{\tau_k}\abs{G(s)}_{\mathcal{L}_2(U,H)}^\alpha \,{\rm d}s\right]
    \leq k^2<\infty, 
\end{align*}
we obtain that $(X_nM)^{\tau_k}$ is a martingale by \cite[Th.\ I:51]{PR}. 

Noting $\int_0^\cdot G\,{\rm d}L_n=\int_0^\cdot G\pi_n \, {\rm d}L$, inequality \eqref{eq:moment_estimate_integral} and equality \eqref{eq:stopped_integral} establish for each $t\geq 0$ that
\begin{align*}
    \E\left[{|(X_nM-XM)^{\tau_k}(t)|}\right]&\leq k\E\left[\left \vert\int_0^{t}\mathbbm{1}_{[0,\tau_k]}G(\pi_n-I)\,{\rm d}L\right \vert \right]\\
    &\leq e_{1,\alpha}k\left(\E\left[\int_0^{T}\abs{G(s)(\pi_n-I)}_{\rm \mathcal{L}_2(U,H)}^\alpha {\rm d}s\right]\right)^{1/\alpha}
    \to 0\quad\text{as } n\to\infty. 
\end{align*}
It follows that the process $(XM)^{\tau_k}$ as a limit of martingales is itself a martingale. Since $X$ is a local martingale according to Lemma \ref{le.local_martingle_property} and $M$ is an arbitrary real-valued continuous martingale, it follows from \cite[Le.\ I.4.14]{JS} that $X$ is purely discontinuous. \\
\textit{Step 3:} For the general case, we fix an orthonormal basis $(e_i)_{i\in\mathbb{N}}$ in $H$ and choose any $i,j\in\mathbb{N}$. Since  $\langle X(t),e_i\rangle=\int_0^t{G^*e_i}\, {\rm d}L $ for every $t\geq 0$,  the polarisation formula for real-valued covariation shows 
\begin{align*}
    \langle\left[\left[X\right]\right], e_i\otimes e_j\rangle&=\left[ \int_0^{\cdot}G^*e_i\,{\rm d}L, \int_0^{\cdot}G^*e_j\,{\rm d}L\right]\\
    &=\frac{1}{2}\left( \left[ \int_0^{\cdot}G^*(e_i+e_j)\,{\rm d}L\right]-\left[ \int_0^{\cdot}G^*e_i\,{\rm d}L\right]-\left[ \int_0^{\cdot}G^*e_j\,{\rm d}L\right]\right).
\end{align*}
Linearity of the integral and binomial formula enable us to conclude 
\begin{align*}
    \sum_{0\leq s\leq t}&{\Delta\langle X(s),e_i\rangle\langle X(s),e_j\rangle}\\
    &=\sum_{0\leq s\leq t}{\Delta\left(\int_0^sG^*e_i\,{\rm d}L\right)\left(\int_0^sG^*e_j\,{\rm d}L\right)}\\
    &=\sum_{0\leq s\leq t}\frac{1}{2}\left({\Delta \left(\int_0^sG^*(e_i+e_j)\,{\rm d}L\right)^2-\Delta\left(\int_0^sG^*e_i\, {\rm d}L\right)^2-\Delta\left(\int_0^sG^*e_j\,{\rm d}L\right)^2}\right).
\end{align*}
The very definition \eqref{def:quadratic_variation_purely_discontinuous} of the continuous part 
leads us to 
\begin{align*}
  \langle\left[\left[X\right]\right]^c, e_i\otimes e_j\rangle &=  \langle \left[\left[X\right]\right], e_i\otimes e_j\rangle-\sum_{0\leq s\leq t}{\Delta\langle X(s),e_i\rangle\langle X(s),e_j\rangle}\\
    &=\frac{1}{2}\left( \left[ \int_0^{\cdot}G^*(e_i+e_j)\,{\rm d}L\right]^c-\left[ \int_0^{\cdot}G^*e_i\,{\rm d}L\right]^c-\left[ \int_0^{\cdot}G^*e_j\,{\rm d}L\right]^c\right).
\end{align*}
Since \textit{Step 2} guarantees that the processes  $\int_0^\cdot G^*(e_i+e_j)\,{\rm d}L$, $ \int_0^\cdot G^*e_i\,{\rm d}L$
and $\int_0^\cdot G^*e_i\,{\rm d}L$ are purely discontinuous, it follows that 
$\langle\left[\left[X\right]\right]^c, e_i\otimes e_j\rangle=0$ for all $i,j\in\N$ which completes the proof. 
\end{proof}

\section{Strong Itô formula}

In this section, we establish an Itô formula for processes that are given by a differential driven by a standard symmetric $\alpha$-stable cylindrical L\'evy process $L$ for $\alpha\in (1,2)$ and are of the form
\begin{align}
dX(t)=F(t)\,{\rm d}t+G(t)\,{\rm d}L(t)  \quad \text{for }t\in [0,T],
\label{eq:differential}
\end{align}
where $F\colon\Omega\times [0,T]\rightarrow H$, $G\colon\Omega\times [0,T]\rightarrow\mathcal{L}_2(U,H)$ are predictable and satisfy
\begin{align}
    \int_0^T\norm{F(t)}+\norm{G(t)}_{\mathcal{L}_2(U,H)}^\alpha \,{\rm d}t<\infty \quad \text{a.s.}
    \label{condition:differential_integrability}
\end{align}
We denote by $\mathcal{C}_b^{2}(H)$ the space of continuous functions $f\colon H \rightarrow \mathbb{R}$ having bounded first and second Fr\'echet derivatives, which are denoted by $Df$ and $D^2f$, respectively. 

\begin{Thm}\label{thm:ito_formula}
Let $X$ be a stochastic process of the form \eqref{eq:differential}. 
It follows for each $f\in\mathcal{C}_b^{2}(H)$ and $t\in [0,T]$ that 
\begin{align*} 
f(X(t))&=
f(X(0))+ \int_0^t \langle Df(X(s-)), F(s)\rangle\, \d s +\int_0^t\langle G(s)^*Df(X(s-)), \cdot\rangle\, {\rm d}L(s)+M_f(t)\\
& + \int_0^t \int_H \Big(f(X(s-)+g)-f(X(s-))-\langle Df(X(s-)),g \rangle\Big) \, \left(\lambda\circ G(s)^{-1}\right)({\rm d}g)\, \d s,  
  \end{align*}
where $M_f:=\big(M_f(t):\, t\in [0,T]\big)$  is a local martingale defined by  
\begin{equation*}
\begin{aligned}[b]
    M_f(t):=\int_0^t\int_H \big(f(X({s-})+h)-f(X(s-))-\langle Df(X(s-)),h\rangle\big)\, (\mu^X-\nu^X)(\d s, \d h).
\end{aligned}
\label{thm:ito_formula_local_martingale}
\end{equation*}
\end{Thm}

\begin{Lemma}\label{lemma_estimate_inner_integral}
Let $\lambda$ be the cylindrical L\'evy measure of an $\alpha$-stable cylindrical L\'evy process for
$\alpha \in (1,2)$. Then we have for each $f\in\mathcal{C}_b^{2}(H)$, $h\in H$, and $\Phi\in\mathcal{L}_2(U,H)$ that
\begin{equation*}
        \int_{H}\big \vert f(h+g)-f(h)-\langle Df(h), g\rangle\big \vert  \,\left(\lambda\circ \Phi^{-1}\right)({\rm d}g)\\
\leq d_\alpha^1\left(2\abs{Df}_{\infty}+\frac{1}{2}\abs{D^2f}_{\infty}\right)\norm{\Phi}_{\mathcal{L}_2(U,H)}^\alpha,
\end{equation*}
where $d_\alpha^1$ is a constant depending only on $\alpha$ as defined in Inequality \eqref{estimate_alpha_norm}.
\end{Lemma}
\begin{proof}
Taylor's remainder theorem in the integral form, see \cite[Th.\ 5.8]{amman}, and Inequality \eqref{estimate_alpha_norm} imply
\begin{align} \label{proof:estimate_inner_integral}
        \int_{\overline{B}_H}&{\big \vert f(h+g)-f(h)-\langle Df(h),g\rangle\big\vert} \left(\lambda\circ \Phi^{-1}\right)({\rm d}g)\notag \\
&=\int_{\overline{B}_H}\left \vert\int_0^1\langle D^2f(h+\theta g)g, g\rangle(1-\theta)\,\d\theta \right \vert\, \left(\lambda\circ \Phi^{-1}\right)({\rm d}g)\notag\\
&\leq\abs{D^2f}_{\infty}\int_{\overline{B}_H}\left({\int_0^1\norm{g}^2(1-\theta)\,{\rm d}\theta}\right)\,\left(\lambda\circ\Phi^{-1}\right)({\rm d}g)\notag\\
&=\frac{1}{2}\norm{D^2f}_{\infty}\int_{\overline{B}_H}{\norm{g}^2}\, \left(\lambda\circ\Phi^{-1}\right)({\rm d}g)\notag\\
&\leq d_\alpha^1\frac{1}{2}\norm{D^2f}_{\infty}\norm{\Phi}_{\mathcal{L}_2(U,H)}^\alpha.
\end{align}
Similarly,  Taylor's remainder theorem in the integral form and Inequality \ref{estimate_alpha_norm} show
\begin{align}\label{proof:estimate_inner_integral-2}
\int_{\overline{B}_H^c}{\left \vert f(h+g)-f(h)\right \vert}\, \left(\lambda\circ\Phi^{-1}\right)({\rm d}g)&=\int_{\overline{B}_H^c}\left \vert\int_0^1{\langle Df(h+\theta g),g\rangle}d\theta\right \vert\, \left(\lambda\circ\Phi^{-1}\right)({\rm d}g)\notag\\
&\leq\abs{Df}_{\infty}\int_{\overline{B}_H^c}\left(\int_0^1\abs{g}\,{\rm d}\theta \right)\, \left(\lambda\circ\Phi^{-1}\right)({\rm d}g)\notag\\
&\leq d_\alpha^1\abs{Df}_\infty\abs{\Phi}_{\mathcal{L}_2(U,H)}^\alpha.
\end{align}
Another application of Inequality \ref{estimate_alpha_norm} shows
\begin{align}\label{proof:estimate_inner_integral-3}
\int_{\overline{B}_H^c}{\left \vert\langle Df(h),g\rangle\right \vert}\,\left(\lambda\circ\Phi^{-1}\right)({\rm d}g)&\leq\abs{Df}_{\infty}\int_{\overline{B}_H^c}\abs{g}\, \left(\lambda\circ\Phi^{-1}\right)({\rm d}g)\notag\\
&\leq d_\alpha^1\abs{Df}_\infty\abs{\Phi}_{\mathcal{L}_2(U,H)}^\alpha.
\end{align}
Combining inequalities \eqref{proof:estimate_inner_integral} to \eqref{proof:estimate_inner_integral-3}
completes the proof.
\end{proof}

\begin{proof}[Proof of Theorem \ref{thm:ito_formula}]
The stochastic process $X$ given by \eqref{eq:differential} is purely discontinuous as it is the sum of a finite-variation process and a purely discontinuous process according to Proposition \ref{prop:purely_discontinuous}. The Itô formula in \cite[Th.\ 27.2]{ME} takes for all $t\in [0,T]$ the form
\begin{align}
    \d f(X(t))&=\langle Df(X(t-)),\cdot\rangle \, \d X(t)\nonumber\\
    &\qquad + \int_H \big( f(X({t-})+h)-f(X(t-))-\langle Df(X(t-)),h\rangle\big) \, \mu^X({\rm d}t,{\rm d}h).
\label{eq:ito_metivier}
\end{align}
One can show by approximating with simple integrands that
\begin{align*}
    \langle Df(X(t-)), \cdot\rangle \,\d X(t)=\langle Df(X(t-)), F(t)\rangle \, {\rm d}t+\langle G(t)^*Df(X(t-)), \cdot\rangle\, {\rm d}L(t),
\end{align*}
where both integrals are well defined since \eqref{condition:differential_integrability} guarantees 
\begin{align*}
    \int_0^T|\langle Df(X(t-)),F(t)\rangle|+&\norm{\langle G(t)^*Df(X(t-)),\cdot\rangle}_{\mathcal{L}_2(U,\mathbb{R})}^\alpha \,{\rm d}t \\
    \leq &\norm{Df}_{\infty}\int_0^T\norm{F(t)}\, {\rm d}t+\norm{Df}_{\infty}^\alpha\int_0^T\norm{G(t)}_{\mathcal{L}_2(U,H)}^\alpha \, {\rm d}t<\infty \quad \text{a.s.}
\end{align*}
The definition of the compensator $\nu^X$ and Lemma \ref{lemma_estimate_inner_integral} imply
\begin{equation}
\begin{aligned}[b]
  &   \E\left[\int_0^T\int_H\big\vert f(X({s-})+h)-f(X(s-))-\langle Df(X(s-)),h\rangle\big\vert\, \mu^X({\rm d}s,{\rm d}h)\right]\\
  &\qquad =\E\left[\int_0^T\int_H \big\vert f(X({s-})+h)-f(X(s-))-\langle Df(X(s-)),h\rangle \big\vert\,\nu^X({\rm d}s,{\rm d}h)\right]\\
  &\qquad \leq d_\alpha^1\left(2\abs{Df}_{\infty}+\frac{1}{2}\abs{D^2f}_{\infty}\right)\E\left[\int_0^T\norm{G(s)}_{\mathcal{L}_2(U,H)}^\alpha \,{\rm d}s\right].
\end{aligned}
\label{proof_ito_formula_1}
\end{equation}
The stopping times
$
    \tau_n:=\inf\left\{ t>0:\int_0^t\abs{G(s)}_{\mathcal{L}_2(U,H)}^\alpha\, {\rm d}s \geq n \right\}\wedge T 
$
satisfy $\tau_n\to T$ as $n\to\infty$ by (\ref{condition:differential_integrability}).  Since inequality \eqref{proof_ito_formula_1} guarantees for all $n\in\N$ that 
\begin{align*}
    \E\left[\int_0^{T\land\tau_n}\right.\left.\int_H\vert f(X({s-})+h)-f(X(s-))-\langle Df(X(s-)),h\rangle\vert\mu^X({\rm d}s,{\rm d}h)\right]<\infty,
\end{align*}
Proposition \ref{equ_char_of_comp} shows that $M_f$ is a local martingale.  This concludes the proof, since the claimed formula is just a different form of \eqref{eq:ito_metivier}.
\end{proof}

\section{Mild solutions for stochastic evolution equations}

We recall that $U$ and $H$ are separable Hilbert spaces with norms $\norm{\cdot}$ and $L$ is a standard symmetric $\alpha$-stable cylindrical $(\mathcal{F}_t)$-Lévy process in $U$ with $\alpha\in (1,2)$. In this section we consider the mild solution of the stochastic evolution equation: 
\begin{equation}
\begin{aligned}[b]
    {\rm d}X(t)&=\big( AX(t) + F(X(t))\big)\,{\rm d}t + G(X(t-))\,{\rm d}L(t)\qquad\text{for }t\in [0,T],\\
    X(0)&=x_0, 
\end{aligned}
    \label{eq:equation}
\end{equation}
where $A$ is a generator of a $C_0$-semigroup $\left( S(t)\right)_{t\geq 0}$ in $H$, $x_0$ is an $\mathcal{F}_0$-measurable $H$-valued random variable, $F\colon H\to H$ and $G\colon H\to\mathcal{L}_2(U,H)$ are measurable mappings and $T>0$. 

\begin{Definition}
An $H$-valued predictable process $X$ is a mild solution to \eqref{eq:equation} if
\begin{align*}
    X(t)=S(t)x_0+\int_0^tS(t-s)F(X(s)){\rm d}s+\int_0^tS(t-s)G(X(s-)){\rm d}L(s) \quad \text{for every }
    t\in [0,T].
\end{align*}
\end{Definition}
We work under the following assumptions:
\begin{itemize}
    \item [(A1)] 
    The $C_0$-semigroup $\left( S(t)\right)_{t\geq 0}$ is compact, analytic and a semigroup of contractions and $0$ is an element of the resolvent set of $A$.
    \item [(A2)] The mapping $F$ is Lipschitz and bounded, i.e.\ there exists $K_F\in (0,\infty)$ such that
    \begin{align}
        \norm{F(h_1)-F(h_2)}\leq K_F\norm{h_1-h_2}, \quad \norm{F(h)}\leq K_F
        \label{ass:coef_F}
    \end{align}
    for every $h_1, h_2, h\in H$.
    \item [(A3)] The mapping $G$ is Lipschitz and bounded, i.e.\ there exists $K_G\in (0,\infty)$ such that
    \begin{align}
    \norm{G(h_1)-G(h_2)}_{\mathcal{L}_2(U,H)}\leq K_G\norm{h_1-h_2}, \quad \norm{G(h)}_{\mathcal{L}_2(U,H)}\leq K_G
    \label{ass:coef_G_0}
    \end{align}
    for every $h_1, h_2, h\in H$.
    \item [(A4)] The initial condition $x_0$ has finite $p$-th moment for every $p<\alpha$.
\end{itemize}

\begin{Remark}\label{re:compact-fractional}
We shall use  the notation $D^\delta:=\text{Dom}((-A)^\delta)$  for the domain of the fractional generator $(-A)^\delta$ for $\delta\in [0,1]$, and equip $D^\delta$ with the norm  $\abs{h}_\delta:=\abs{(-A)^\delta h}$. It follows from Assumption (A1) that the embedding of Hilbert spaces $D^{\delta}\hookrightarrow D^{\gamma}$ is dense and compact for every $ 0\leq \gamma<\delta\leq 1$, cf. \cite[Cor. 3.8.2]{bergh2012interpolation}.
\end{Remark}
\begin{Remark}\label{re:analytic}
Assumption (A1) implies, cf. \cite[p.\ 289]{HI},  that for every $\delta\geq 0$ there exists a $c_{\delta}\in (0,\infty)$ depending only on $\delta$ such that
\begin{align}
    \norm{S(t)}_{\mathcal{L}(H,D^\delta)}\leq c_{\delta} t^{-\delta}
 \qquad\text{for every $t>0$.}    \label{eq:analytic}
\end{align}
\end{Remark}
\begin{Remark}
By considering the cases $\norm{h_1-h_2}\le 1$ and  $\norm{h_1-h_2}>1$ separately, we conclude from Assumptions (A2) and (A3) that there exist $K_F,K_G\in (0,\infty)$ such that for any $\beta\in (0,1)$ we have
    \begin{align}
    \norm{F(h_1)-F(h_2)}\leq K_F\norm{h_1-h_2}^\beta, \quad \norm{F(h)}\leq K_F,
    \label{ass:coef_F2}
    \end{align}
and
    \begin{align}
    \norm{G(h_1)-G(h_2)}_{\mathcal{L}_2(U,H)}\leq K_G\norm{h_1-h_2}^\beta, \quad \norm{G(h)}_{\mathcal{L}_2(U,H)}\leq K_G,
    \label{ass:coef_G}
    \end{align}
    for every $h_1, h_2, h\in H$.
\end{Remark}

The first main theorem of this article is the following existence result, which also includes properties on the  path regularity of the solution. 

\begin{Thm}
Under the assumptions (A1)-(A4), there exists a mild solution $X$ to \eqref{eq:equation}. The mild solution $X$ is an element of $\mathcal{C}([0,T],L^p(\Omega,H))$ for every $p<\alpha$ and has c\`adl\`ag paths in $H$. 
\label{thm:existence_mild}
\end{Thm}

We will obtain the solution to \eqref{eq:equation} by using the Yosida approximations. For this purpose, we define $ R_n = n\left(n{\rm I}-A\right)^{-1}$ for $n\in\mathbb{N}$ and denote by $X_n$ the mild solution to
\begin{equation}
\begin{aligned}[b]
    {\rm d}X_n(t) &= \big( AX_n(t) + R_nF(X_n(t))\big)\,{\rm d}t + R_nG(X_n(t-))\,{\rm d}L(t),\\
    X_n(0) &=R_nx_0.
\end{aligned}
    \label{eq:equation_approx}
\end{equation}
Existence of the mild solution $X_n$ to \eqref{eq:equation_approx} with c\'adl\'ag paths is guaranteed by  \cite[Th.\ 12]{KR}.

\begin{Remark}
We recall that under Assumption (A1) we have for all $\delta\in [0,1]$ that
    \begin{align*}
        \norm{R_n}_{\mathcal{L}(D^\delta)}\leq 1, \quad n\in\mathbb{N}.
    \end{align*}
This follows from the fact, that if an operator commutes with $A$ then it commutes with $A^\gamma$, see e.g.\ \cite[Pr.\ 3.1.1]{HA}, which enables us to conclude     for every $n\in\mathbb{N}$ that
\begin{align*}
        \norm{R_n}_{\mathcal{L}(D^\gamma)}
        =\sup_{\norm{(-A)^\gamma h}\leq 1}\norm{n (n - A)^{-1}(-A)^\gamma h}
        \leq\sup_{\norm{h}\leq 1}\norm{n (n-A)^{-1}h} 
        =\norm{R_n}_{\mathcal{L}(H)}. 
    \end{align*}
Since $(S(t))_{t\ge 0}$ is a contraction semigroup,  Theorem 3.5 in \cite{EN} guarantees $\norm{R_n}_{\mathcal{L}(H)}\le 1$ for all $n\in{\mathbb N}$.
\label{re.yosida_boundedness}
\end{Remark}

The solution to \eqref{eq:equation} will be constructed as a limit of $X_n$ in $\mathcal{C}([0,T],L^p(\Omega,H))$ for an arbitrary but fixed $p<\alpha$. In the first three Lemmata, we establish relative compactness of the Yosida approximation  $\lbrace X_n:n\in\mathbb{N}\rbrace$ in the space $\mathcal{C}([0,T],L^p(\Omega,H))$.

\begin{Lemma} \label{le.tightness}
The set $\lbrace X_n(t):n\in\mathbb{N}\rbrace$ is tight in $H$ for every $t\in [0,T]$.
\end{Lemma}

\begin{proof}
The case $t=0$ follows immediately from the strong convergence of $R_n$. For the case $t\in (0,T]$ we first prove that for every $1\leq q<\alpha$, and $0\leq\delta<1/\alpha$ we have
\begin{align}
    \sup_{n\in\mathbb{N}}\E\left[\norm{X_n(t)}_\delta^q\right]<\infty.
    \label{proof:relative_compactness_fixed_t1}
\end{align}
Applying Hölder's inequality and inequality \eqref{eq:moment_estimate_integral} shows for every $n\in\mathbb{N}$ 
that 
\begin{align*}
    \E\left[\norm{X_n(t)}_\delta^q\right]
    &\leq3^{q-1}\left(\E\left[\norm{S(t)R_nx_0}_\delta^q\right]+t^{q-1}\E\left[\int_0^t\norm{S(t-s)R_nF(X_n(s))}_\delta^q\, {\rm d}s\right]\right.\\
    &\qquad \qquad\left.+e_{q, \alpha}\left(\E\left[\int_0^t\norm{S(t-s)R_nG(X_n(s))}_{\mathcal{L}_2(U,D^\delta)}^\alpha \, {\rm d}s\right]\right)^{\frac{q}{\alpha}}\right).
\end{align*}
Commutativity of $S$ and $R_n$, Remark \ref{re:analytic} and Remark \ref{re.yosida_boundedness} verify 
\begin{align*}
    \E\left[\norm{S(t)R_nx_0}_\delta^q\right]\leq c_{\delta}^qt^{-q\delta} \sup_{n\in\mathbb{N}}\norm{R_n}_{\mathcal{L}(D^\delta)}^q\E\left[\norm{x_0}^q\right]<\infty.
\end{align*}
Assumption (A2) on boundedness of $F$ together with Remark \ref{re:analytic} and Remark \ref{re.yosida_boundedness} yield
\begin{align*}
    \E\left[\int_0^t\norm{S(t-s)R_nF(X_n(s))}_\delta^q\, {\rm d}s\right]\leq c_{\delta}^q\frac{t^{1-q\delta}}{1-q\delta}\sup_{n\in\mathbb{N}}\norm{R_n}_{\mathcal{L}(D^\delta)}^qK_F^q<\infty.
\end{align*}
Similarly,  Assumption (A3) on boundedness of G implies 
\begin{align*}
    \left(\E\left[ \int_0^t\norm{S(t-s)R_nG(X_n(s))}_{\mathcal{L}_2(U,D^\gamma)}^\alpha \, {\rm d}s\right]\right)^{\frac{q}{\alpha}}\leq c_{\delta}^q\left(\frac{t^{1-\alpha\delta}}{1-\alpha\delta}\right)^{\frac{q}{\alpha}}\sup_{n\in\mathbb{N}}\norm{R_n}_{\mathcal{L}(D^\delta)}^qK_G^q<\infty.
\end{align*}
Combining the above estimates establishes \eqref{proof:relative_compactness_fixed_t1}, which in turn gives the statement of the Lemma. Indeed, choose any $\delta\in (0,1/\alpha)$ and use Markov's inequality and \eqref{proof:relative_compactness_fixed_t1} for $q=1$  to obtain for each $N>0$ that \[\sup_{n\in\mathbb{N}}P\left(\norm{X_n(t)}_\delta>N\right)\leq\frac{c}{N}\]
for some constant $c\in (0,\infty)$.
Since the embedding $D^\delta\hookrightarrow H$ is compact according to Remark \ref{re:compact-fractional}, we obtain tightness of $\lbrace X_n(t):n\in\mathbb{N}\rbrace$ by Prokhorov's theorem.
\end{proof}

\begin{Lemma}
\label{le.relative_compactness_fixed_t_probability}
The sequence $\lbrace X_n(t):n\in\mathbb{N}\rbrace$ is relatively compact in $L_P^0(\Omega,H)$ for every $t\in [0,T]$.
\end{Lemma}
\begin{proof}
    First, we impose the additional assumption that
    \begin{align}
    T12^2e_{2,\alpha}\left(K_F^\alpha+K_G^\alpha\right)<1,
        \label{proof:convergence_prob1}
    \end{align}
    where $K_F, K_G$ come from \eqref{ass:coef_F2}, \eqref{ass:coef_G} and $e_{2,\alpha}$ is defined just below \eqref{eq:moment_estimate_integral}. For $1<p<\alpha$ and $m,n\in\mathbb{N}$ we estimate the $p$-th moment of the difference $X_m(t)-X_n(t)$ by
    \begin{equation}
	\label{eq_estimate_1_2_3}
	\begin{split}
		\E\left[\norm{X_m(t)-X_n(t)}^p\right]
		\leq&6^{p-1}\Bigg(\E\left[\norm{S(t)(R_m-R_n)x_0}^p\right]\\
		&+\E\left[\norm{\int_0^tS(t-s)(R_m-R_n)F(X_m(s)){\rm d}s}^p\right]\\
		&+\E\left[\norm{\int_0^tS(t-s)R_n(F(X_m(s))-F(X_n(s))){\rm d}s}^p\right]\\
		&+\E\left[\norm{\int_0^tS(t-s)(R_m-R_n)G(X_m(s-)){\rm d}L(s)}^p\right]\\
		&+\E\left[\norm{\int_0^tS(t-s)R_n(G(X_m(s-))-G(X_n(s-))){\rm d}L(s)}^p\right]\Bigg).
	\end{split}
\end{equation}
Applying Hölder's inequality, Jensen's inequality, inequality \eqref{eq:moment_estimate_integral}, Remark \ref{re.yosida_boundedness}, (A1) and estimates \eqref{ass:coef_F2} and \eqref{ass:coef_G} we obtain
\begin{align*}
    &\E\left[\norm{\int_0^tS(t-s)R_n(F(X_m(s))-F(X_n(s))){\rm d}s}^p\right.\left.\vphantom{\int_0^t}\right]\\
    &\qquad\leq T^{p-\frac{p}{\alpha}}\left(\E\left[\int_0^t\norm{S(t-s)R_n(F(X_m(s))-F(X_n(s)))}^\alpha{\rm d}s\right]\right)^{\frac{p}{\alpha}}\\
    &\qquad\leq T^{p-\frac{p}{\alpha}}K_F^p\left(\E\left[\int_0^t\norm{X_m(s)-X_n(s)}^p{\rm d}s\right]\right)^{\frac{p}{\alpha}},
\end{align*}
and, similarly, 
\begin{align*}
    &\E\left[\norm{\int_0^tS(t-s)R_n(G(X_m(s-))-G(X_n(s-))){\rm d}L(s)}^p\right.\left.\vphantom{\int_0^t}\right]\\
    &\qquad\leq e_{p,\alpha}K_G^p\left(\E\left[\int_0^t\norm{X_m(s)-X_n(s)}^p{\rm d}s\right]\right)^{\frac{p}{\alpha}}.
\end{align*}
Raising both sides of \eqref{eq_estimate_1_2_3} to the power of $\alpha/p$ shows for each $t\in [0,T]$ that 
\begin{align}\label{eq.Gronwall-pre}
	u_{n,m,p}(t)\leq u_{n,m,p}^0+w_{p}\int_0^t\left(u_{n,m,p}(s)\right)^{\frac{p}{\alpha}}{\rm d}s,
\end{align}
where we use 
\begin{align*}
    u_{n,m,p}(t)&:=\left(\E\left[\norm{X_m(t)-X_n(t)}^p\right]\right)^{\frac{\alpha}{p}},\\
    u_{n,m,p}^0&:=5^{\frac{\alpha}{p}-1}6^{\frac{\alpha}{p}(p-1)}\Bigg(\sup_{t\in [0,T]}\norm{S(t)}_{\mathcal{L}(H)}^\alpha\left(\E\left[\norm{(R_m-R_n)x_0}^p\right]\right)^{\frac{\alpha}{p}}\\
    &\hphantom{hhhhhhhhhhhhhh}+\left(\sup_{t\in [0,T]}\E\left[\norm{\int_0^tS(t-s)(R_m-R_n)F(X_m(s)){\rm d}s}^p\right]\right)^{\frac{\alpha}{p}}\\
    &\hphantom{hhhhhhhhhhhhhh}+\left(\sup_{t\in [0,T]}\E\left[\norm{\int_0^tS(t-s)(R_m-R_n)G(X_m(s-)){\rm d}L(s)}^p\right]\right)^{\frac{\alpha}{p}}\Bigg),\\
    w_{p}&:=5^{\frac{\alpha}{p}-1}6^{\frac{\alpha}{p}(p-1)}\left(T^{\alpha-1}K_F^\alpha+e_{p,\alpha}^{\frac{\alpha}{p}}K_G^\alpha\right).
\end{align*}
It follows from \eqref{eq.Gronwall-pre} by a version of Gronwall's inequality  \cite[Th. 2]{willett1965discrete} that
\begin{align}
    u_{n,m,p}(t)& 
     \leq 2^{\frac{\alpha}{\alpha-p}}u_{n,m,p}^0+\left(2\frac{\alpha-p}{\alpha}tw_{p}\right)^{\frac{\alpha}{\alpha-p}}.
    \label{eq_estimate_1_2_3_c}
\end{align}
If we show for each $t\in [0,T]$  that
\begin{align}
\lim_{p\to\alpha-}\left(2\frac{\alpha-p}{\alpha}tw_p\right)^{\frac{\alpha}{\alpha-p}}=0,
    \label{eq_estimate_1_2_3_d}
\end{align}
and, for any $1<p<\alpha$, 
\begin{align}
\lim_{m,n\to\infty}u_{n,m,p}^0=0,
    \label{eq_estimate_1_2_3_e}
\end{align}
then, for each $\epsilon\in(0,1)$, we can find some $p^*\in(1,\alpha)$ such that
$
    2^{\frac{\alpha}{\alpha-p^*}}\left(\frac{\alpha-p^*}{\alpha}tw_{p^*}\right)^{\frac{\alpha}{\alpha-p^*}}<\frac{\epsilon^{\alpha+1}}{2},
$
and some $N\in\mathbb{N}$ such that for all $m,n\geq N$
$
    2^{\frac{\alpha}{\alpha-p^*}}u_{n,m,p^*}^0\leq\frac{\epsilon^{\alpha+1}}{2}.
$
This completes the proof since we obtain for any $m,n\geq N$ that
\begin{align*}
    P\left(\norm{X_{m}(t)-X_{n}(t)}\geq\epsilon\right)\leq \frac{1}{\epsilon^{p^*}}\E\left[\norm{X_{m}(t)-X_{n}(t)}^{p^*}\right]
    \leq\epsilon^{\frac{p^*}{\alpha}}. 
\end{align*}
It remains to establish \eqref{eq_estimate_1_2_3_d} and \eqref{eq_estimate_1_2_3_e}, which we first do under the additional assumption \eqref{proof:convergence_prob1}:

\textit{Argument for \eqref{eq_estimate_1_2_3_d}:} Since  $e_{p,\alpha}=\frac{\alpha}{\alpha-p}e_{2,\alpha}^{p/\alpha}$ for a constant $e_{2,\alpha}\in (0,\infty)$ independent of $\alpha$, we obtain $T^{\alpha-1}\leq e_{p,\alpha}^{\frac{\alpha}{p}}$ for $p$ sufficiently close to $\alpha$. In this case, we obtain 
\begin{align*}
    \left(2\frac{\alpha-p}{p}tw_p\right)^{\frac{\alpha}{\alpha-p}}
    &\leq 5^{\frac{\alpha}{p}}\left(2\frac{\alpha-p}{\alpha}t6^2\left(\frac{\alpha-p}{\alpha}\right)^{-\frac{\alpha}{p}}e_{2,\alpha}\left(K_F^\alpha+K_G^\alpha\right)\right)^{\frac{\alpha}{\alpha-p}}\\
    &\leq 5^\alpha\left(\frac{\alpha-p}{\alpha}\right)^{-2}\left(T12^2e_{2,\alpha}\left(K_F^\alpha+K_G^\alpha\right)\right)^{\frac{\alpha}{\alpha-p}}.
\end{align*}
Since we assume  \eqref{proof:convergence_prob1}, applying L'Hospital's rule verifies \eqref{eq_estimate_1_2_3_d}.

\textit{Argument for \eqref{eq_estimate_1_2_3_e}:}
Let $p \in(1,\alpha)$ be fixed. By strong convergence of $R_n$ and Lebesgue's dominated convergence theorem we have
\begin{align}
    \lim_{m,n\to\infty}\E\left[\norm{(R_m-R_n)x_0}^p\right]=0.
    \label{proof:convergence_prob1a}
\end{align}
Hölder's inequality, strong convergence of $R_n$, Lemma \ref{le.uniform_convergence} and Lebesgue's dominated convergence theorem we have
\begin{equation}
\begin{split}
    \lim_{m,n\to\infty}&\left(\sup_{t\in [0,T]}\E\left[\norm{\int_0^tS(t-s)(R_m-R_n)F(X_m(s)){\rm d}s}^p\right]\right)\\
    \leq&T^{p-1}\sup_{t\in [0,T]}\norm{S(t)}_{\mathcal{L}(H)}^p\lim_{m,n\to\infty}\left(\E\left[\int_0^T\norm{(R_m-R_n)F(X_m(s))}^p{\rm d}s\right]\right)=0,
\end{split}
    \label{proof:convergence_prob1b}
\end{equation}
where the assumptions of Lemma \ref{le.uniform_convergence} are satisfied by boundedness of $F$, see \eqref{ass:coef_F}, and tightness of $\lbrace F(X_n(s)), m\in\mathbb{N}\rbrace$ implied by Lemma \ref{le.tightness} and continuity of $F$. Similarly, using inequality \eqref{eq:moment_estimate_integral} and \eqref{ass:coef_G}, we obtain
\begin{equation}
\begin{split}
    \lim_{m,n\to\infty}&\left(\sup_{t\in [0,T]}\E\left[\norm{\int_0^tS(t-s)(R_m-R_n)G(X_m(s-)){\rm d}L(s)}^p\right]\right)\\
    \leq&e_{p,\alpha}\sup_{t\in [0,T]}\norm{S(t)}_{\mathcal{L}(H)}^p\lim_{m,n\to\infty}\left(\E\left[\int_0^T\norm{(R_m-R_n)G(X_m(s))}^\alpha{\rm d}s\right]\right)^{\frac{p}{\alpha}}=0,
\end{split}
    \label{proof:convergence_prob1c}
\end{equation}
which establishes  \eqref{eq_estimate_1_2_3_e}. 

Thus, we have proved the lemma under the additional assumption \eqref{proof:convergence_prob1}. For establishing the general case, we fix a time $T_0 \le (12^2e_{2,\alpha}\left(K_F^\alpha+K_G^\alpha\right))^{-1}$. 
If $t\in [0,T_0]$, relative compactness in $L_P^0(\Omega,H)$ follows from the previous arguments. If $t\in (T_0, 2T_0]$ we write
\begin{align*}
    X_m(t)-X_n(t)
    &=S\left(t-T_0\right)\left(X_m\left(T_0\right)-X_n\left(T_0\right)\right)\\
    &\qquad+ \int_{T_0}^t S(t-s)(R_mF(X_m(s)-R_nF(X_n(s)))\, {\rm d}s\\
    &\qquad+\int_{T_0}^t S(t-s)(R_mG(X_m(s-)-R_nG(X_n(s-)))\, {\rm d}L(s).
\end{align*}
Since our choice of $T_0$ implies  that $\lbrace S(t-T_0)X_n(T_0),\, n\in\mathbb{N}\rbrace$ is relatively compact in $L_P^0(\Omega,H)$, and that  $\left(t-T_0\right)12^2e_{2,\alpha}\left(K_F^\alpha+K_G^\alpha\right)<1$, we can use the same argument as before to obtain relative compactness of $(X_n(t))_{n \in \mathbbm{N}}$ in $L_P^0(\Omega,H)$ for each $t \in (T_0, 2T_0]$. Using a standard induction argument, we can now cover intervals of arbitrary length. This concludes the proof of the general case.
\end{proof}

We now step from relative compactness of $\lbrace X_n(t): n\in\mathbb{N}\rbrace$ in $L_P^0(\Omega,H)$ for fixed time $t$ to relative compactness of the processes $\lbrace X_n: n\in\mathbb{N}\rbrace$ using the Arzelà–Ascoli Theorem.

\begin{Lemma}
The collection $\{X_n: n\in\mathbb{N}\}$ is relatively compact in $\mathcal{C}([0,T],L^p(\Omega,H))$ for any $0<p<\alpha$.
\label{le.relative_compactness}
\end{Lemma}
\begin{proof}
We consider the case $1<p<\alpha$ as the case $p\leq 1$ follows from the fact that relative compactness in $\mathcal{C}([0,T],L^{p}(\Omega,H))$ implies relative compactness in $\mathcal{C}([0,T],L^{p'}(\Omega,H))$ for $p>p'$. In light of the Arzelà–Ascoli Theorem, cf.\ e.g.\ \cite[Th.\ 7.17]{kelley}), it suffices to show that
\begin{enumerate}
    \item [\rm (a)] $\{X_n(t): n\in\mathbb{N}\}\subset L^p(\Omega,H)$ is relatively compact for each $t\in [0,T]$;
    \item [\rm (b)] $\{X_n: n\in\mathbb{N}\}\subset \mathcal{C}([0,T],L^p(\Omega,H))$ is equicontinuous.
\end{enumerate}
The claim in (a) follows from \cite[Cor. 3.3]{DM} by Lemmata \ref{le.tightness}, \ref{le.relative_compactness_fixed_t_probability} and the fact that Equation \eqref{proof:relative_compactness_fixed_t1} with $\delta=0$ and any $q \in (p,\alpha)$ implies via the Vallee-Poussin Theorem \cite[Th.\ II.22]{dellacherie_meyer_1978} that the collection $\{X_n(t): n\in\mathbb{N}\}$ is $p$-uniformly integrable and bounded in $L^p(\Omega,H)$. Hence, it remains only to prove (b).  To that end, we take $t \in [0,T]$ and $h \in (0,T-t]$, and estimate
\begin{align}\label{eq.right-cont-5terms}
& \norm{X_n(t+h)-X_n(t)}^p \nonumber\\
&\leq 5^{p-1}\Bigg(\norm{\left(S(h)-I\right)S(t)R_nx_0}^p +\norm{\int_t^{t+h}S(t+h-s)R_nF(X_n(s))\, {\rm d}s}^p\nonumber\\
& \quad+\norm{\int_t^{t+h}S(t+h-s)R_nG(X_n(s-))\, {\rm d}L(s)}^p+\norm{\int_0^t\left(S(h)-I\right)S(t-s)R_nF(X_n(s))\,{\rm d}s}^p\notag\\
&\quad +\norm{\int_0^t\left(S(h)-I\right)S(t-s)R_nG(X_n(s-))\,{\rm d}L(s)}^p\Bigg).
\end{align}
Commutativity of $R_n$ and $S$ and contractivity of $S$ implies 
\begin{equation}\label{eq.right-cont-first-term}
 E\left[\norm{\left(S(h)-I\right)S(t)R_nx_0}^p\right]
  \le \sup_{n\in\mathbb{N}}\norm{R_n}_{\mathcal{L}(H)}^p
  \E\left[\norm{\left(S(h)-I\right)x_0}^p\right].
\end{equation}
Applying  Hölder's inequality, boundedness of $F$ in Assumption (A2) and contractivity of $S$ we get
\begin{equation}
    \E\left[\norm{\int_t^{t+h}S(t+h-s)R_nF(X_n(s))\, {\rm d}s}^p\right]
    \le h^p \sup_{n\in\mathbb{N}}\norm{R_n}_{\mathcal{L}(H)}^pK_F^p. 
\end{equation}
We conclude from Inequality \eqref{eq:moment_estimate_integral} by using boundedness of $G$ in Assumption (A3) and contractivity of $S$ that
\begin{align}
&\E\left[\norm{\int_t^{t+h}S(t+h-s)R_nG(X_n(s-))\,{\rm d}L(s)}^p\right] \\
&\qquad \le e_{p,\alpha}\left(\E\left[\int_t^{t+h}\norm{S(t+h-s)G(X_n(s))}_{\mathcal{L}_2(U,H)}^\alpha\, {\rm d}s\right]\right)^{p/\alpha}
\le e_{p,\alpha}\sup_{n\in\mathbb{N}}\norm{R_n}_{\mathcal{L}(H)}^pK_G^ph^{p/\alpha}. \notag
\end{align}
It follows from Lemma \ref{le.tightness} that $\{X_n(s):\, n\in\N\}$ is tight in $H$ for every $s\in [0,t]$.  Lemma \ref{le.uniform_convergence} implies 
\begin{align*}
\lim_{h\searrow 0}  \sup_{n\in\mathbb{N}}\E\left[\norm{\left(S(h)-I\right)S(t-s)R_nF(X_n(s))}^p\right]=0.
\end{align*}
Lebesgue's dominated convergence theorem shows
\begin{align}
&\lim_{h\searrow 0}\int_0^t \sup_{n\in\N}\E\left[\norm{\left(S(h)-I\right)S(t-s)R_nF(X_n(s))}^p\right]\,{\rm d}s=0.
\end{align}
In the same way, after applying Inequality \eqref{eq:moment_estimate_integral}, we obtain
 from Lemma  \ref{le.uniform_convergence}  
\begin{equation}\label{eq.right-cont-last-term}
 \lim_{h\searrow 0}\sup_{n\in\mathbb{N}}\E\left[\norm{\int_0^t\left(S(h)-I\right)S(t-s)R_nG(X_n(s-))\,{\rm d}L(s)}^p\right] = 0.
\end{equation}
Applying \eqref{eq.right-cont-first-term} - \eqref{eq.right-cont-last-term} to Inequality \eqref{eq.right-cont-5terms} shows uniform continuity from the right. 
Similar arguments establish uniform continuity from the left, which 
proves (b), and thus completes the proof.

\end{proof}

\begin{proof}[Proof of Theorem \ref{thm:existence_mild}]
It is enough to consider the case $p\geq\alpha\beta$ where $\beta$ is the Hölder exponent from \eqref{ass:coef_G}. Lemma \ref{le.relative_compactness} guarantees that there is a  subsequence $(n_k)_{k=1}^\infty$ such that
\begin{align}
    \lim_{k \rightarrow \infty}\sup_{t\in [0,T]}\E\left[\abs{X_{n_k}(t)-Z(t)}^p\right]= 0
    \label{proof:rel_comp_contra}
\end{align}
for some $Z\in\mathcal{C}([0,T], L^p(\Omega, H))$.
The proof will be complete if we show that $Z$ is a mild solution to \eqref{eq:equation}. We conclude for each $k\in\mathbb{N}$ and $t\in [0,T]$ from Lipschitz continuity of $F$ and Hölder continuity of $G$ in  \eqref{ass:coef_F} and \eqref{ass:coef_G} and contractivity of $S$ by applying Hölder's inequality and Inequality 
\eqref{eq:moment_estimate_integral} that
\begin{align*}
    \E&\left[\abs{Z(t)-S(t)x_0-\int_0^tS(t-s)F(Z(s))\,{\rm d}s-\int_0^tS(t-s)G(Z(s-))\,{\rm d}L(s)}^p\right]\\
    &\leq(1\wedge 3^{p-1})\Bigg(\E\left[\abs{Z(t)-X_{n_k}(t)}^p\right]+\E\left[\abs{\int_0^tS(t-s)\big(F(Z(s))-F(X_{n_k}(s))\big)\,{\rm d}s}^p\right]\\
    &\hphantom{PHANTOMMM}+\E\left[\abs{\int_0^tS(t-s)\big(G(Z(s-))-G(X_{n_k}(s-))\big)\,{\rm d}L(s)}^p\right]\Bigg)\\
    &\leq (1\wedge 3^{p-1}) \Bigg(\E\left[\abs{Z(t)-X_{n_k}(t)}^p\right]+T^{p-1}\E\left[\int_0^t\abs{S(t-s)\big(F(Z(s))-F(X_{n_k}(s))\big)}^p\,{\rm d}s\right]\\
    &\hphantom{PHANTOMMM}+\left.e_{p,\alpha}\left(\E\left[\int_0^t\abs{S(t-s)\big(G(Z(s))-G(X_{n_k}(s))\big)}_{\mathcal{L}_2(U,H)}^\alpha\,{\rm d}s\right]\right)^{p/\alpha}\right)\\
    &\leq (1\wedge 3^{p-1})\Bigg(\E\left[\norm{Z(t)-X_{n_k}(t)}^p\right]+T^{p-1}K_F^p\sup_{t\in [0,T]}\norm{S(t)}_{\mathcal{L}(H)}^p\E\left[\int_0^t\norm{Z(s)-X_{n_k}(s)}^p\,{\rm d}s\right]\\
    &\hphantom{PHANTOMMM}+\left.e_{p,\alpha} K_G^p\sup_{t\in [0,T]}\norm{S(t)}_{\mathcal{L}(H)}^p\left(\E\left[\int_0^t\norm{Z(s)-X_{n_k}(s)}_{\mathcal{L}_2(U,H)}^{\alpha\beta}\,{\rm d}s\right]\right)^{p/\alpha}\right)\\
    &\leq(1\wedge 3^{p-1})\left(\left(1+T^pK_F^p\right)\sup_{t\in [0,T]}\E\left[\norm{Z(t)-X_{n_k}(t)}^p\right]\right.\\
    &\hphantom{PHANTOMMM}+\left.e_{p,\alpha}K_G^pT^{p/\alpha}\sup_{t\in [0,T]}\left(\E\left[\norm{Z(t)-X_{n_k}(t)}^p\right]\right)^\beta\right)
\end{align*}
As the last line tends to $0$ as $k\to\infty$ by \eqref{proof:rel_comp_contra}, it follows that $Z$ is  a mild solution to \eqref{eq:equation}.

It remains to establish that $Z$ has c\`adl\`ag paths, but this follows immediately from the following corollary as $X_n$ has  c\`adl\`ag paths.
\end{proof}

At the end of this section, we present a stronger convergence result for Yosida approximations that not only completes the proof of Theorem \ref{thm:existence_mild} but also turns out to be useful in applications as will be seen in the following sections.

\begin{Corollary}\label{le.yosida_convergence_uniform}
For all $0<p<\alpha$ there exists a subsequence $(X_{n_k})_{k \in \mathbb{N}}$ of the Yosida approximations, which converges to a solution to \eqref{eq:equation} both in $\mathcal{C}([0,T],L^p(\Omega,H))$ and uniformly on $[0,T]$ almost surely.
\end{Corollary}
\begin{proof}
Lemma \ref{le.relative_compactness} enables us to choose a subsequence $(X_n)_{n\in\N}$ of the Yosida approximations which converges in $\mathcal{C}([0,T],L^p(\Omega,H))$ to the mild solution $X$. To prove  almost sure convergence, we fix an arbitrary $r>0$ and estimate
\begin{align}\label{eq.cor-Yosida-sure}
    \begin{split}
   &P\left(\sup_{t\in [0,T]}\norm{X(t)-X_n(t)}>r\right)
   \leq P\left(\sup_{t\in [0,T]}\norm{S(t)(I-R_n)x_0}>\frac{r}{3}\right)\\
  &\qquad\qquad   +P\left(\sup_{t\in [0,T]}\norm{\int_0^tS(t-s)\big(F(X(s))-R_nF(X_n(s))\big)\,{\rm d}s}>\frac{r}{3}\right) \\
    & \qquad\qquad +P\left(\sup_{t\in [0,T]}\norm{\int_0^tS(t-s)\big(G(X(s-))-R_nG(X_n(s-))\big)\,{\rm d}L(s)}>\frac{r}{3}\right). 
   \end{split}
\end{align}
For the following arguments, we define $m:=\sup_{t\in [0,T]}\norm{S(t)}_{\mathcal{L}(H)}$.
As $I-R_n$ converges to zero strongly as $n\to\infty$ we obtain
\begin{align*}
P\left(\sup_{t\in [0,T]}\norm{S(t)(I-R_n)x_0}>\frac{r}{3}\right)\leq P\left(m\norm{(I-R_n)x_0}>\frac{r}{3}\right)\to 0.
\end{align*}
For estimating the second term in \eqref{eq.cor-Yosida-sure}, we 
apply Markov's inequality and Lipschitz continuity of $F$ in (A2) to obtain 
\begin{align*}
    &P\left(\sup_{t\in [0,T]}\norm{\int_0^tS(t-s)\big(F(X(s))-R_nF(X_n(s))\big)\,{\rm d}s}>\frac{r}{3}\right)\\
    &\leq  P\left(m\int_0^T\norm{F(X(s))-R_nF(X_n(s))}{\rm d}s>\frac{r}{3}\right)\\
    &\leq  P\left(\int_0^T\norm{(I-R_n)F(X(s))}\,{\rm d}s>\frac{r}{6m}\right)
    + P\left(\int_0^T\norm{R_n(F(X(s))-F(X_n(s)))}\,{\rm d}s>\frac{r}{6m}\right)\\
    &\leq  \frac{6m}{r}\E\left[\int_0^T\norm{(I-R_n)F(X(s))}\,{\rm d}s\right]
    + \frac{6m}{r}\E\left[\int_0^T\norm{R_n(F(X(s))-F(X_n(s)))}\,{\rm d} s\right]\\
    &\leq \frac{6m}{r}\E\left[\int_0^T\norm{(I-R_n)F(X(s))}\, {\rm d}s\right]
    +\frac{6m}{r} \left(\sup_{n\in\mathbb{N}}\norm{R_n}_{\mathcal{L}(H)}\right)TK_F\sup_{t\in [0,T]}\E\left[\norm{X(t)-X_n(t)}\right].
\end{align*}
We conclude from the last inequality by Lebesgue's dominated convergence theorem and 
convergence of $X_n$ to $X$ in $\mathcal{C}([0,T],L^1(\Omega,H))$ that 
\begin{align*}
    \lim_{n\to\infty} P\left(\sup_{t\in [0,T]}\norm{\int_0^tS(t-s)(F(X(s))-R_nF(X_n(s)))\,{\rm d}s}>\frac{r}{3}\right)=0.
\end{align*}
To estimate the last term in \eqref{eq.cor-Yosida-sure}, 
we apply the dilation theorem for contraction semigroups, see \cite[Th.\ I.8.1]{NAFO}):  there exists a $C_0$-group $(\hat{S}(t))_{t\in\R}$ of unitary operators $\hat{S}(t)$ on a larger Hilbert space $\hat{H}$ in which  $H$ is continuously embedded satisfying  $S(t)=\pi\hat{S}(t)i$ for all $t\ge 0$, where $\pi$ is the projection from $\hat{H}$ to $H$ and $i$ is the continuous embedding of $H$ into $\hat{H}$. Thus, if we denote $m=\sup_{t\in[0,T]}\norm{\pi\hat{S}(t)}_{\mathcal{L}(\hat{H},H)}$, $k=\sup_{s\in [-T,0]}\norm{\hat{S}(s)i}_{\mathcal{L}(H,\hat{H})}$, we may estimate using Markov's inequality, Inequality \eqref{eq:moment_estimate_integral} and Hölder continuity of $G$ in \eqref{ass:coef_G}
\begin{align*}
    &P\left(\sup_{t\in [0,T]}\norm{\int_0^tS(t-s)\left(G(X(s-))-R_nG(X_n(s-))\right){\rm d}L(s)}>\frac{r}{3}\right)\\
    &\leq P\left(\sup_{t\in [0,T]}\norm{\int_0^tS(t-s)\left(I-R_n\right)G(X(s-)){\rm d}L(s)}>\frac{r}{6}\right)\\
    &\qquad\qquad+P\left(\sup_{t\in [0,T]}\norm{\int_0^tS(t-s)R_n\left(G(X(s-))-G(X_n(s-))\right){\rm d}L(s)}>\frac{r}{6}\right)\\
    &=P\left(\sup_{t\in [0,T]}\norm{\int_0^t\pi\hat{S}(t-s)i\left(I-R_n\right)G(X(s-)){\rm d}L(s)}>\frac{r}{6}\right)\\
    &\qquad\qquad+P\left(\sup_{t\in [0,T]}\norm{\int_0^t\pi\hat{S}(t-s)iR_n\left(G(X(s-))-G(X_n(s-))\right){\rm d}L(s)}>\frac{r}{6}\right)\\
    &=P\left(\sup_{t\in [0,T]}\norm{\pi\hat{S}(t)\int_0^t\hat{S}(-s)i\left(I-R_n\right)G(X(s-)){\rm d}L(s)}>\frac{r}{6}\right)\\
    &\qquad\qquad+P\left(\sup_{t\in [0,T]}\norm{\pi\hat{S}(t)\int_0^t\hat{S}(-s)iR_n\left(G(X(s-))-G(X_n(s-))\right){\rm d}L(s)}>\frac{r}{6}\right)\\
    &\leq P\left(\sup_{t\in [0,T]}\norm{\int_0^t\hat{S}(-s)i\left(I-R_n\right)G(X(s-)){\rm d}L(s)}_{\hat{H}}>\frac{r}{6m}\right)\\
    &\qquad\qquad+P\left(\sup_{t\in [0,T]}\norm{\int_0^t\hat{S}(-s)iR_n\left(G(X(s-))-G(X_n(s-))\right){\rm d}L(s)}_{\hat{H}}>\frac{r}{6m}\right)\\
    &\leq\frac{6m}{r}\E\left[\sup_{t\in [0,T]}\norm{\int_0^t\hat{S}(-s)i\left(I-R_n\right)G(X(s-)){\rm d}L(s)}_{\hat{H}}\right] \\
    &\qquad\qquad+\frac{6m}{r}\E\left[\sup_{t\in [0,T]}\norm{\int_0^t\hat{S}(-s)iR_n\left(G(X(s-))-G(X_n(s-))\right){\rm d}L(s)}_{\hat{H}}\right]\\
    &\leq e_{1,\alpha}\frac{6m}{r}\left(\E\left[\int_0^T\norm{\hat{S}(-s)i\left(I-R_n\right)G(X(s-))}_{\mathcal{L}_2(U,\hat{H})}^\alpha{\rm d}s\right]\right)^{1/\alpha}\\
    &\qquad\qquad+e_{1,\alpha}\frac{6m}{r}\left(\E\left[\int_0^T\norm{\hat{S}(-s)iR_n\left(G(X(s-))-G(X_n(s-))\right)}_{\mathcal{L}_2(U,\hat{H})}^\alpha{\rm d}s\right]\right)^{1/\alpha}\\
    &\leq e_{1,\alpha}\frac{6m}{r}k\left(\E\left[\int_0^T\norm{\left(I-R_n\right)G(X(s-))}_{\mathcal{L}_2(U,\hat{H})}^\alpha{\rm d}s\right]\right)^{1/\alpha}\\
    &\qquad\qquad+e_{1,\alpha}\frac{6m}{r}kK_GT^{1/\alpha}\left(\sup_{t\in [0,T]}\E\left[\norm{(X(t))-X_n(t)}^{\alpha\beta}\right]\right)^{1/\alpha}\\
\end{align*}
We conclude from the last inequality by Lebesgue's dominated convergence, strong convergence of $R_n$ to $I$, boundedness $G$ in \eqref{ass:coef_G} and 
convergence of $X_n$ to $X$ in $\mathcal{C}([0,T],L^{\alpha\beta}(\Omega,H))$ that 
\begin{align*}
    \lim_{n\to\infty} P\left(\sup_{t\in [0,T]}\norm{\int_0^tS(t-s)(G(X(s-))-R_nG(X_n(s-)))\,{\rm d}L(s)}>\frac{r}{3}\right)=0,
\end{align*}
We have shown that all the terms on the right hand side of \eqref{eq.cor-Yosida-sure} converge to zero as $n$ tends to infinity which gives uniform convergence of $X_n$ to $X$ in probability on $[0,T]$.
This concludes the proof, since uniform convergence in probability implies the existence of a desired subsequence.
\end{proof}

\section{Moment boundedness for evolution equations}

In this section, we  investigate stability properties of the solution 
for the stochastic evolution equation \eqref{eq:equation} by applying the It{\^o}'s formula derived in Theorem \ref{thm:ito_formula}. More precisely, we shall derive conditions on the coefficients such that 
the mild solution $X$ is ultimately exponentially bounded in the $p$-th moment, that is there exist constants $m_1,m_2,m_3 > 0$ such that
\begin{align*}
\E\left[\norm{X(t)}^p\right]\leq m_1e^{-tm_2}\E\left[\norm{x_0}^p\right]+m_3 \quad \text{for all } t    \geq 0.
\end{align*}

Recall that $\mathcal{C}_b^2(H)$ denotes the space of continuous real-valued functions defined on $H$ with bounded first and second Fréchet derivatives. In what follows, our goal is to derive a Lyapunov-type criterion  using the following operator on $\mathcal{C}_b^2(H)$:
\begin{equation}
\begin{aligned}[b]
    \mathcal{L}f(h)=\langle &Df(h),Ah+F(h)\rangle\\
    &+\int_H \big( f(h+g)-f(h)-\langle Df(h),g \rangle\big)\,\left(\lambda\circ G(h)^{-1}\right)({\rm d}g), \quad h\in D^1
\end{aligned}
\label{def.L}
\end{equation}
for $f\in\mathcal{C}_b^2(H)$. Note that the right hand side of \eqref{def.L} is well defined by Lemma \ref{lemma_estimate_inner_integral}. We can now state the main result of this section, the following general moment boundedness criterion.

\begin{Thm}
Let $p\in (0,1)$ be fixed and $V$ be a function in $\mathcal{C}_b^2(H)$ satisfying for some constants $\beta_1, \beta_2, \beta_3, k_1, k_3>0$ the inequalities
\begin{align}
   &  \beta_1\abs{h}^p-k_1\leq V(h)\leq \beta_2\abs{h}^p \quad\text{for all } h\in H,
    \label{ass:lyapunov_functinon_bound}\\
   &\mathcal{L}V(h)\leq -\beta_3 V(h)+k_3 \quad\text{for all } h\in D^1.
    \label{ass:lyapunov_functinon_generator}
\end{align}
Then the solution $X$ to \eqref{eq:equation} is exponentially ultimately bounded in the $p$-th moment:
\begin{align*}
    \E\left[\abs{X(t)}^p\right]\leq \frac{\beta_2}{\beta_1}e^{-\beta_3 t}\E\left[\norm{x_0}^p\right]+\frac{1}{\beta_1}\left(k_1+\frac{k_3}{\beta_3}\right).
\end{align*}
\label{thm:criterion}
\end{Thm}

Before we prove Theorem \ref{thm:criterion} we demonstrate its application by deriving conditions for moment boundedness in terms of the coefficients of Equation \eqref{eq:equation}. 
\begin{Corollary}
Suppose that there exists $\epsilon>0$ such that
\begin{align*}
    \langle Ah+F(h),h \rangle\leq -\epsilon\abs{h}^2 \quad\text{for all } h\in D^1,
\end{align*}
then the solution to \eqref{eq:equation} is exponentially ultimately bounded in the $p$-th moment for every $p\in (0,1)$.
\end{Corollary}
\begin{proof}
Fix $p\in (0,1)$ and let $\zeta$ be a function in $\mathcal{C}^2([0,\infty))$ satisfying $\zeta (x)=x^{p/2}$ for $x\geq 1$ and $\zeta(x)\leq 1$ for $x<1$. By defining $V(h)=\zeta (\norm{h}^2)$ for all $h\in H$, we obtain  $V\in\mathcal{C}_b^2(H)$ 
and 
\begin{align*}
  V(h)=\abs{h}^p  \quad \text{for all }h\in \overline{B}_H^c \quad\text{ and }\quad
  0\leq V(h)\leq1  \quad \text{for all }h\in \overline{B}_H.
\end{align*}
It follows that \eqref{ass:lyapunov_functinon_bound} holds with $\beta_1=\beta_2=k_1=1$. We show that \eqref{ass:lyapunov_functinon_generator} also holds. By the definition of $V$,  it follows for each $h\in D^1\cap \overline{B}_H^c$ that
\begin{align*}
    \langle DV(h), Ah+F(h)\rangle
    =p\abs{h}^{p-2}\langle h,Ah+F(h)\rangle
    \le  -\epsilon p\abs{h}^p= -\epsilon p V(h). 
\end{align*}
For $h\in D^1\cap \overline{B}_H$, one obtains by boundedness of $F$ in Assumption (A2) that 
\begin{align*}
    \langle DV(h), Ah+F(h)\rangle
    \le \abs{DV}_\infty\left(\abs{A}_{\mathcal{L}(D^1)}+K_F\right). 
\end{align*}
Since Lemma \ref{lemma_estimate_inner_integral}  together 
with boundedness of $G$ in Assumption (A3) implies for each $h\in H$ that 
\begin{align*}
\int_H \big(V(h+g)-V(h)-\langle DV(h),g\rangle\big) \,\left(\lambda\circ G^{-1}(h)\right) ({\rm d}g)
\le d_\alpha^1\left(2\abs{DV}_\infty+\frac{1}{2}\abs{D^2V}_\infty\right)K_G^\alpha,
\end{align*}
we have verified Condition \eqref{ass:lyapunov_functinon_generator}. 
\end{proof}

In the remaining of this section, we prove Theorem \ref{thm:criterion} using the Yosida approximations established in the previous sections.
For this purpose, let $X_n$ denote the mild solution to the approximating equations \eqref{eq:equation_approx} for each $n\in\N$. We may assume
due to Corollary \ref{le.yosida_convergence_uniform}, by passing to a subsequence if necessary, that $X_n$ converges to the solution $X$
of \eqref{eq:equation} uniformly almost surely on $[0,T]$. In what follows, we will routinely pass on to a subsequence without changing the indices.

\begin{Proposition}\label{pro.weak-is-strong}
The mild solution $X_n$ of \eqref{eq:equation_approx} is a strong solution attaining values in $D^1$, that is, for each $t\in [0,T]$, it satisfies 
\begin{align*}
    X_n(t)=R_nx_0+\int_0^t \big( AX_n(s) + R_nF(X_n(s))\big) \,\d s+\int_0^t R_nG(X_n(s-))\, {\rm d}L(s).
\end{align*}

\end{Proposition}
\begin{proof}
Our argument will follow closely the proof of \cite[Th.2]{AGMRS}. As mild solution, $X_n$ satisfies
\begin{align}
X_n(t)=S(t)R_nx_0+\int_0^tS(t-s)R_nF(X_n(s))\,\d s+\int_0^tS(t-s)R_nG(X_n(s-))\,{\rm d}L(s).
    \label{def:mild_solution_yosida}
\end{align}
The process $X_n$  is  $\left(\mathcal{F}_t\right)$-measurable with c\`adl\`ag paths and attains values in $D^1$.  First, we obtain from \eqref{def:mild_solution_yosida} by interchanging integrals and $A\in\mathcal{L}(D^1)$ for $t\geq 0$  that
\begin{equation}
\begin{aligned}[b]
    AX_n(t)=AS(t)R_nx_0+\int_0^tAS(t-s)&)R_nF(X_n(s))\,{\rm d}s\\
    &+\int_0^tAS(t-s)R_nG(X_n(s-))\, {\rm d}L(s).
\end{aligned}
\label{proof:strong_solution_0}
\end{equation}
Each term on the right hand side of \eqref{proof:strong_solution_0} is almost surely Bochner integrable. Indeed, integrability of the first term is immediate from the uniform boundedness principle. For the second term, boundedness of $F$ in Condition (A2) and commutativity of $S$ and $R_n$ implies 
\begin{align*}
    \int_0^t\int_0^s&\norm{AS(s-r)R_nF(X_n(r))}_1\,{\rm d}r\, {\rm d}s\\
    &\leq \norm{A}_{\mathcal{L}(D^1)}\norm{R_n}_{\mathcal{L}(H,D^1)}\int_0^t\int_0^s\norm{S(s-r)F(X_n(s))}\,{\rm d}r\, {\rm d}s<\infty \quad \text{a.s.}
\end{align*}
Almost sure Bochner integrability of the stochastic integral in \eqref{proof:strong_solution_0} follows from boundedness of $G$ in  Assumption(A3), commutativity of $S$ and $R_n$, and Theorem \ref{thm.fubini} via the estimate
\begin{align*}
& \E\left[\int_0^t\int_0^s\norm{AS(s-r)R_nG(X_n(r))}_{\mathcal{L}_2(U,D^1)}^\alpha\, {\rm d}r\, {\rm d}s\right]\\
    &\qquad \leq\norm{A}_{\mathcal{L}(D^1)}^\alpha\norm{R_n}_{\mathcal{L}(H,D^1)}^\alpha\E\left[\int_0^t\int_0^s\norm{S(s-r)G(X_n(r))}_{\mathcal{L}_2(U,H)}^\alpha\, {\rm d}r\, {\rm d}s\right]<\infty.
\end{align*}
Integrating both sides of \eqref{proof:strong_solution_0} results in the equality 
\begin{align*}
    \int_0^tAX_n(s)\, {\rm d}s&=\int_0^tAS(s)R_nx_0 \,{\rm d}s+\int_0^t\int_0^sAS(s-r)R_nF(X_n(r))\,{\rm d}r\,{\rm d}s\\
    &\qquad +\int_0^t\int_0^sAS(s-r)R_nG(X_n(r-))\,{\rm d}L(r)\,{\rm d}s.
\end{align*}
Applying Fubini's theorems, see Theorem \ref{thm.fubini} for the stochastic version, and the equality $\int_0^t AS(s)R_nh\,\d s = S(t)R_nh-R_nh$ for all $h\in H$ enable us to conclude
\begin{align*}
 \int_0^tAX_n(s)\,{\rm d}s&=\int_0^tAS(s)R_nx_0\, {\rm d}s+\int_0^t\int_r^tAS(s-r)R_nF(X_n(r))\,{\rm d}s\, {\rm d}r\\
    &\qquad+\int_0^t\int_r^tAS(s-r)R_nG(X_n(r-))\,{\rm d}s\,{\rm d}L(r)\\
    &=S(t)R_nx_0-R_nx_0 + \int_0^tS(t-r)R_nF(X_n(r))\,{\rm d}r-\int_0^tR_nF(X_n(r))\,{\rm d}r\\
    &\qquad +\int_0^tS(t-r)R_nG(X_n(r-))\,{\rm d}L(r)-\int_0^tR_nG(X_n(r-))\,{\rm d}L(r)\\
    &=X_n(t)-R_nx_0-\int_0^tR_nF(X_n(r))\,{\rm d}r-\int_0^tR_nG(X_n(r-))\,{\rm d}L(r), 
\end{align*}
which verifies $X_n$ as a strong solution to \eqref{eq:equation}.
\end{proof}

We denote by $\mathcal{L}_n$ the usual generator associated with the Yosida approximations $X_n$, $n\in\mathbb{N}$,  defined for $f\in\mathcal{C}_b^2(H)$ and $h\in D^1$ by
\begin{equation}
    \begin{aligned}[b] \mathcal{L}_n f(h)=\langle &Df(h),Ah+R_nF(h)\rangle\\
    &+\int_H \big(f(h+R_ng)-f(h)-\langle Df(h),R_ng \rangle\big)\, \left(\lambda\circ G(h)^{-1}\right)({\rm d}g). 
    \end{aligned}
    \label{def:generator_yosida}
\end{equation}
The right hand side of \eqref{def:generator_yosida} is well defined by Lemma \ref{lemma_estimate_inner_integral}.
Recall that the counterpart to $\mathcal{L}_n$ for the mild solution $X$ denoted by $\mathcal{L}$ was introduced in \eqref{def.L}. The generators $\mathcal{L}_n$ and $\mathcal{L}$ are related by the following crucial convergence result.
\begin{Lemma}\label{le.generator_convergence}
Let $(X_n)_{n\in\N}$ be solutions of \eqref{eq:equation_approx} which a.s.\ converges uniformly to the solution of \eqref{eq:equation}. 
It follows for each $f\in\mathcal{C}_b^2(H)$ that
\begin{align*}
    \lim_{n \rightarrow \infty}\E\left[\int_0^T\bigg|\mathcal{L}_n f(X_n(s))\right.&\left.-\mathcal{L}f(X_n(s))\bigg|\, {\rm d}s\right]=0.
\end{align*}
\end{Lemma}
\begin{proof}
We obtain for each $h\in D^1$ that 
\begin{align}\label{eq.L_n-L}
&\vert \mathcal{L} f(h)-\mathcal{L}_nf(h)\vert
        \leq\abs{Df}_\infty\abs{\left( I-R_n\right)F(h)}\notag \\
        & +\int_{\overline{B}_H}\vert f(h+g)-f(h+R_ng)-\langle Df(h),\left(I-R_n\right)g\rangle\vert \,\left(\lambda\circ G(h)^{-1}\right)({\rm d}g)\\
        & +\int_{\overline{B}_H^c} \!\vert f(h+g)-f(h+R_ng)\vert  \left(\lambda\circ G(h)^{-1}\right)({\rm d}g)
          +\int_{\overline{B}_H^c}\!\vert\langle Df(h),\left(I-R_n\right)g\rangle\vert  \left(\lambda\circ G(h)^{-1}\right)({\rm d}g).\notag
\end{align}
Taylor's remainder theorem in the integral form implies
\begin{align*}
 &\int_{\overline{B}_H}\vert f(h+g)-f(h+R_ng)-\langle Df(h),\left(I-R_n\right)g\rangle\vert \,\left(\lambda\circ G(h)^{-1}\right)({\rm d}g)\\
 &\qquad \leq \int_{\overline{B}_H}\int_0^1\vert \langle D^2f(h+\theta\left(I-R_n\right)g)\left(I-R_n\right)g, \left(I-R_n\right)g\rangle(1-\theta)\,\vert{\rm d}\theta\, \left(\lambda\circ G(h)^{-1}\right)({\rm d}g)\\
&\qquad\leq \frac{1}{2}\abs{D^2f}_{\infty}\int_{\overline{B}_H}\abs{\left(I-R_n\right)g}^2\,\left(\lambda\circ G(h)^{-1}\right)({\rm d}g).
 \end{align*}
In the same way, we obtain 
\begin{align*}
\int_{\overline{B}_H^c}\!\vert f(h+g)-f(h+R_ng)\vert \left(\lambda\circ G(h)^{-1}\right)({\rm d}g)
\leq \abs{Df}_\infty\int_{\overline{B}_H^c} \!\abs{ \left( I-R_n \right)g} \left(\lambda\circ G(h)^{-1}\right)({\rm d}g)
\end{align*}
and also
\begin{align*}
    \int_{\overline{B}_H^c}\!\vert\langle Df(h),\left(I-R_n\right)g\rangle\vert  \left(\lambda\circ G(h)^{-1}\right)({\rm d}g)\leq \abs{Df}_\infty\int_{\overline{B}_H^c} \!\abs{ \left( I-R_n \right)g} \left(\lambda\circ G(h)^{-1}\right)({\rm d}g)
\end{align*}
Applying the last three estimates to \eqref{eq.L_n-L} and taking expectation on both sides, it follows from Inequality \eqref{estimate_alpha_norm} and for each $n\in\N$ that
\begin{align*}
    &\E\left[\int_0^T\left\vert \mathcal{L}_n f(X_n(s)) -\mathcal{L}f(X_n(s))\right\vert{\rm d}s \right] \\
    &\quad\leq\abs{Df}_\infty\E\left[\int_0^T\abs{\left( I-R_n\right)F(X_n(s))}{\rm d}s\right]
    +c\E\left[\int_0^T\abs{\left( I-R_n\right)G(X_n(s))}_{\mathcal{L}_2(U,H)}^\alpha{\rm d}s\right],
\end{align*}
where $c:=d_\alpha^1\left(2\abs{Df}_\infty+\frac{1}{2}\abs{D^2f}_\infty\right)$. 

To complete the proof, it remains to show that both
\begin{align}
&\lim_{n\to\infty}\E\left[\int_0^T\abs{\left( I-R_n\right)F(X_n(s))}{\rm d}s\right]=0,
    \label{proof:generator_convergence_1}\\
&\lim_{n\to\infty}\E\left[\int_0^T\abs{\left( I-R_n\right)G(X_n(s))}_{\mathcal{L}_2(U,H)}^\alpha{\rm d}s\right]=0.
    \label{proof:generator_convergence_2}
\end{align}
Let $t\in [0,T]$ be arbitrary but fixed, and recall that we chose $X_n(t)$  almost surely convergent and thus $\lbrace X_m(t)(\omega):m\in\mathbb{N} \rbrace\subset H$ is compact for almost all $\omega\in\Omega$. Strong convergence of $I-R_n$ to zero, see  \cite[Le.\ 3.4]{EN}, continuity of $F$ and $G$ and the fact that continuous mapping converging pointwise to a continuous mapping converge uniformly over compacts together imply for each $t\in [0,T]$ that, almost surely, we obtain
\begin{align*}
&\lim_{n\to\infty}    \abs{\left( I-R_n\right)F(X_n(t))}\leq\lim_{n \rightarrow \infty}\sup_{m\in\mathbb{N}}\abs{\left( I-R_n\right)F(X_m(t))}=0, \\
&\lim_{n\to\infty}      \abs{\left( I-R_n\right)G(X_n(t))}_{\mathcal{L}_2(U,H)}^\alpha\leq\lim_{n \rightarrow \infty} \sup_{m\in\mathbb{N}}\abs{\left( I-R_n\right)G(X_m(t))}_{\mathcal{L}_2(U,H)}^\alpha=0.
\end{align*}
Since the boundedness conditions in (A2) and (A3) guarantee 
\begin{align*}
    \abs{\left( I-R_n\right)F(X_n(t))} &\leq\left(\sup_{n\in\mathbb{N}}\abs{I-R_n}_{\mathcal{L}(H)}\right)K_F \quad \text{a.s.,}\\
    \abs{\left( I-R_n\right)G(X_n(t))}_{\mathcal{L}_2(U,H)}^\alpha
    &\leq\left(\sup_{n\in\mathbb{N}}\abs{I-R_n}_{\mathcal{L}(H)}^\alpha\right) K_G^\alpha \quad \text{a.s.}
\end{align*}
an application of Lebesgue's dominated convergence theorem verifies \eqref{proof:generator_convergence_1} and \eqref{proof:generator_convergence_2}, which completes the proof.
\end{proof}

\begin{proof}[Proof of Theorem \ref{thm:criterion}]
Let $(X_n)_{n\in\N}$ be the solutions of \eqref{eq:equation_approx}.
Because of Corollary \ref{le.yosida_convergence_uniform}, we can assume 
that $(X_n)_{n\in\N}$  \ converges uniformly to the solution of \eqref{eq:equation} a.s. Proposition \ref{pro.weak-is-strong} enables us to apply the Itô formula in Theorem \ref{thm:ito_formula}  to $X_n$, which results in 
\begin{align}
& V(X_n(t))=V(X_n(0)+\int_0^t\mathcal{L}_nV(X_n(s)){\rm d}s+\int_0^t\langle G(X_n(s-))^*R_n^*DV(X_n(s-)), \cdot\rangle\, {\rm d}L(s)\notag \\
    &\quad +\int_0^t\int_HV(X_n(s-)+h)-V(X_n(s-))-\langle DV(X_n(s-)),h \rangle(\mu^{X_n}-\nu^{X_n})\, ({\rm d}s, \d h)
    \label{proof:criterion_1}
\end{align}
almost surely for all $t\geq 0$. 
Applying the product formula to the real-valued semi-martingale $V(X_n(\cdot))$ and the function $t\mapsto e^{\beta_3t}$ and taking expectations on both sides of \eqref{proof:criterion_1} shows
\begin{align}
    e^{\beta_3 t}\E\left[V(X_n(t))\right]=&\E\left[V(X_n(0))\right]+\E\left[\int_0^te^{\beta_3 s}\left(\beta_3 V(X_n(s))+\mathcal{L}_nV(X_n(s)) \right){\rm d}s\right].
    \label{proof:criterion_2}
\end{align}
Here, we used the fact that the last two  integrals in \eqref{proof:criterion_1} define martingales, and thus have expectation zero. This follows from the observation that they are local martingales according to Lemma \ref{le.local_martingle_property} and Theorem \ref{thm:ito_formula} and are uniformly bounded. The latter is guaranteed by the
boundedness of $G$ in (A3), since 
\begin{align*}
    \E\left[\int_0^t\abs{\langle G(X_n(s))^*R_n^*DV(X_n(s)), \cdot\rangle}_{\mathcal{L}_2(U,\mathbb{R})}^\alpha\,{\rm d}s\right]
    &=\E\left[\int_0^t\abs{ G(X_n(s))^*R_n^*DV(X_n(s))}^\alpha\,{\rm d}s\right]\\
    &\leq \abs{R_n}_{\mathcal{L}(H)}^\alpha\abs{DV}_\infty^\alpha TK_G^\alpha<\infty,
\end{align*}
and similarly, by using Lemma \ref{lemma_estimate_inner_integral}, 
\begin{align*}
    \E\left[\int_0^t\right.&\left.\int_H\vert V (X_n(s-)+h)-V(X_n(s-))-\langle DV(X_n(s-),h\rangle\vert\,\nu^{X_n}({\rm d}s,{\rm d}h)\right]\\
    &\leq d_\alpha^1\left(2\abs{DV}_\infty+\frac{1}{2}\abs{D^2V}_\infty\right)\abs{R_n}_{\mathcal{L}(H)}^\alpha\E\left[\int_0^t\abs{G(X_n(s))}_{\mathcal{L}_2(U,H)}^\alpha\,{\rm d}s\right]<\infty.
\end{align*}
The first term on the right hand side in \eqref{proof:criterion_2} is finite since
\begin{align*}
    \E\left[V(X_n(0))\right]\leq\beta_2\norm{R_n}_{\mathcal{L}(H)}^p\E\left[\norm{x_0}^p\right]<\infty.
\end{align*}
The same holds for the second term, which can be shown 
using the same arguments as in the proof of Lemma \ref{le.generator_convergence}.
By applying  Inequality \eqref{ass:lyapunov_functinon_generator} to \eqref{proof:criterion_2}, we conclude
\begin{align*}
    e^{\beta_3 t}\E\left[[V(X_n(t))\right] 
    &\leq \E\left[V(X_n(0))\right]+\E\Bigg[\int_0^te^{\beta_3 s}\bigg(-\mathcal{L}V(X_n(s))+\mathcal{L}_nV(X_n(s))+k_3\bigg){\rm d}s\Bigg]\\
    &\leq \E\left[V(X_n(0))\right]+e^{\beta_3 T}\E\Bigg[\int_0^t\bigg| \mathcal{L}_nV(X_n(s))-\mathcal{L}V(X_n(s))\bigg|{\rm d}s\Bigg]+\frac{k_3}{\beta_3}\left( e^{\beta_3 t}-1\right).
\end{align*}
Lemma \ref{le.generator_convergence}  together with Fatou's lemma implies
\begin{align*}
    \E\left[V(X (t))\right]
    \leq\liminf_{n\to\infty}\E\left[V(X_n (t))\right]
\leq e^{-\beta_3 t}\E\left[V(x_0)\right]+\frac{k_3}{\beta_3}.
\end{align*}
Applying Assumption \eqref{ass:lyapunov_functinon_bound}
completes the proof.
\end{proof}

\section{Mild Itô formula}

In this section, we prove an Itô formula for mild solutions of Equation \eqref{eq:equation} and mappings $f\in\mathcal{C}_b^2(H)$ such that the second derivative $D^2f$ is not only continuous but satisfies
\begin{align}
    \lim_{n \rightarrow \infty}\norm{g_n-g}=0 \implies \lim_{n \rightarrow \infty}\sup_{h\in \overline{B}_H}\norm{D^2f(g_n+h)-D^2f(g+h)}_{\mathcal{L}(H)}=0.
    \label{ass:f_uniform_continuity}
\end{align}
The subspace of all these functions is denoted by $\mathcal{C}_{b,u}^2(H)$.
\begin{Thm}[It{\^o} formula for mild solutions]
A mild solution $X$ of \eqref{eq:equation} satisfies for each $f\in \mathcal{C}_{b,u}^2(H)$
and $t\geq 0$ that
    \begin{align}   \label{eq:ito_formula_mild}
            f(X(t))&=f(x_0) +\int_0^t \langle G(X(s-)^*Df(X(s-))), \cdot\rangle \,{\rm d}L(s)\\
            &\qquad+\int_0^t\int_H \big(f(X(s-)+h)-f(X(s-))-\langle Df(X(s-)),h\rangle \big)\, \left(\mu^X-\nu^X\right) ({\rm d}s,{\rm d}h)\notag\\
            &\qquad+\lim_{n\to\infty}\left(\int_0^t\langle Df(X_n(s)), AX_n(s)\rangle\, {\rm d}s\right)
            +\int_0^t\langle Df(X(s)), F(X(s)) \rangle \, {\rm d}s\notag\\
            &\qquad+\int_0^t\int_H \big(f(X(s)+h)-f(X(s))-\langle Df(X(s)),h\rangle\big)\, \left(\lambda\circ G(X(s))^{-1}\right)({\rm d}h)\,{\rm d}s, \notag
    \end{align}
    where the limit is taken in $L_P^0(\Omega,\mathbb{R})$.
    \label{prop:ito_formula_mild}
\end{Thm}
\begin{Remark}
    Note that while $X$ may not be a semimartingale, the compensated measure $\mu^X-\nu^X$ in \eqref{eq:ito_formula_mild} still exists as $X$ is both adapted and c\`adl\`ag; see \cite[Chap.\ II]{JS}.
\end{Remark}
\begin{Remark}
Unlike in similar situation with the driving noise being Gaussian e.g. in \cite{maslowski1995stability} we do not identify the limit in \eqref{eq:ito_formula_mild} as then the imposed assumptions on $f$ are very restrictive. In many applications (see e.g. \cite{AGMRS}), it is enough to identify some bound on
\begin{align*}
    \lim_{n\to\infty}\left(\int_0^t\langle Df(X_n(s)), AX_n(s)\rangle\, {\rm d}s\right)
\end{align*}
which leads to natural assumptions on the generator $A$.
\end{Remark}
We divide the proof of the above theorem in some technical lemmas. To simplify the notation, we introduce the function $T_f\colon H\times H \rightarrow \mathbb{R}$ for $f\in \mathcal{C}_{b,u}^2(H)$ defined by
\begin{align*}
    T_f(g,h)=f(g+h)-f(g)-\langle Df(g),h\rangle, \quad g,h\in H.
\end{align*}

\begin{Lemma}\label{le.weak_convergence_assumption}
Let $\lambda$ be the cylindrial L\'evy measure of $L$. 
It follows for every $f\in\mathcal{C}_b^2(H)$, $\phi\in \mathcal{L}_2(U,H)$ and $g, h\in H$  that 
\begin{align*}
    \int_{H}&\left\vert T_f(g,b)-T_f(h,b)\right\vert\left(\lambda\circ \phi^{-1}\right)({\rm d}b)\\
    &\leq 2d_\alpha^1\norm{\phi}_{\mathcal{L}_2(U,H)}^\alpha\left(\sup_{b\in \overline{B}_H}\norm{D^2f(g+b)-D^2f(h+b)}_{\mathcal{L}(H)}+\norm{D^2f}_\infty\norm{g-h}\right).
\end{align*}
\end{Lemma}
\begin{proof}
Taylor's remainder theorem in the integral form and Inequality \eqref{estimate_alpha_norm} imply
\begin{align}
        \int_{\overline{B}_H}&\left\vert T_f(g,b)-T_f(h,b)\right\vert\,\left(\lambda\circ \phi^{-1}\right)({\rm d}b)\notag\\
        &=\int_{\overline{B}_H}\left\vert \int_0^1 \langle (D^2f(g+\theta b)-D^2f(h+\theta b))b,b\rangle(1-\theta)\,{\rm d}\theta\right\vert\,\left(\lambda\circ \phi^{-1}\right)({\rm d}b)\notag\\
        &\leq\frac{1}{2}\sup_{b\in \overline{B}_H}\norm{D^2f(g+b)-D^2f(h+b)}_{\mathcal{L}(H)}\int_{\overline{B}_H}\norm{b}^2\,\left(\lambda\circ \phi^{-1}\right)({\rm d}b)\notag\\
        &\leq\frac{1}{2}d_\alpha^1\left(\sup_{b\in \overline{B}_H}\norm{D^2f(g+b)-D^2f(h+b)}_{\mathcal{L}(H)}\right)\norm{\phi}_{\mathcal{L}_2(U,H)}^\alpha.
    \label{proof:lemma_weak_convergence_formula1}
\end{align}
A similar argument yields
\begin{equation}
    \begin{aligned}[b]
    \int_{\overline{B}_H^c}&\left\vert T_f(g,b)-T_f(h,b)\right\vert\,\left(\lambda\circ \phi^{-1}\right)({\rm d}b)\\
        &\leq\int_{\overline{B}_H^c}\left\vert \int_0^1\langle Df(g+\theta b)-Df(h+\theta b),b\rangle \,{\rm d}\theta\right\vert\,\left(\lambda\circ \phi^{-1}\right)({\rm d}b)\\
   &\qquad\qquad   +\int_{\overline{B}_H^c}\left\vert\langle Df(g)-Df(h), b\rangle\right\vert\,\left(\lambda\circ \phi^{-1}\right)({\rm d}b)\\
           &\leq\left(\sup_{b\in H}\norm{Df(g+b)-Df(h+b)}
  +\norm{Df(g)-Df(h)}\right)\int_{\overline{B}_H^c}\norm{b}\,\left(\lambda\circ \phi^{-1}\right)({\rm d}b)\\
        &\leq 2d_\alpha^1 \norm{D^2f}_\infty\norm{g-h}\norm{\phi}_{\mathcal{L}_2(U,H)}^\alpha.\\
    \end{aligned}
    \label{proof:lemma_weak_convergence_formula2}
\end{equation}
Combining Inequalities \eqref{proof:lemma_weak_convergence_formula1} and \eqref{proof:lemma_weak_convergence_formula2} completes the proof.
\end{proof}

\begin{Lemma}\label{le.convergence_technical}
Let $(X_n)_{n\in\N}$ be a sequence of c\`adl\`ag processes in $H$ which converges to a  process $X$
both in $\mathcal{C}([0,T], L^p(\Omega,H))$ and uniformly on $[0,T]$ almost surely. Then it follows for any $f\in\mathcal{C}_{b,u}^2(H)$ and $t\in [0,T]$ that
\begin{equation*}
 \lim_{n\to\infty}\int_0^t\int_H T_f(X_n(s-),h)\,\mu^{X_n}({\rm d}s,{\rm d}h)=\int_0^t\int_HT_f(X(s-),h)\,\mu^{X}({\rm d}s,{\rm d}h)\qquad\text{in } L_P^0(\Omega,\mathbb{R}).
\end{equation*}
\end{Lemma}
\begin{proof}
Theorem \ref{th.compensator} guarantees for each $n\in\mathbb{N}$ that
\begin{align}    \label{proof:convergence_technicall}
 &\E\left[\left\vert\int_0^t\int_H\big(T_f(X_n(s-),h)-T_f(X(s-),h)\big)\,\mu^{X_n}({\rm d}s,{\rm d}h)\right\vert\right]\notag\\
    &\qquad\leq\E\left[\int_0^t\int_H\big\vert T_f(X_n(s-),h)-T_f(X(s-),h)\big\vert\,\mu^{X_n}({\rm d}s,{\rm d}h)\right]\notag\\
    &\qquad=\E\left[\int_0^t\int_H\big\vert T_f(X_n(s),h)-T_f(X(s),h)\big\vert\,\nu^{X_n}({\rm d}s,{\rm d}h)\right]\notag\\
    &\qquad=\E\left[\int_0^t\int_H\big\vert T_f(X_n(s),h)-T_f(X(s),h)\big\vert\,\left(\lambda\circ (R_nG(X_n(s-)))^{-1}\right)({\rm d}h)\, {\rm d}s\right].
\end{align}
Since Remark \ref{re.yosida_boundedness} guarantees $c:=2 d_\alpha^1 K_G^\alpha\sup_{n\in\mathbb{N}}\norm{R_n}_{\mathcal{L}(H)}^\alpha<\infty$, we obtain from  
Lemma \ref{le.weak_convergence_assumption} for $P\otimes\rm{Leb}$-a.a.\  $(\omega,s )\in \Omega\times [0,T]$ that
\begin{align*}
 & \int_H\left\vert T_f(X_n(s)(\omega),h)-T_f(X(s)(\omega),h)\right\vert\, \left(\lambda\circ (R_nG(X_n(s-)(\omega)))^{-1}\right)({\rm d}h)\\
& \leq  c \Bigg(\sup_{b\in \overline{B}_H}\norm{D^2f(X_n(s)(\omega)+b)-D^2f(X(s)(\omega)+b)}_{\mathcal{L}(H)} \\
&\qquad\qquad\qquad\qquad   +\norm{D^2f}_\infty \norm{f(X_n(s)(\omega))- f(X(s)(\omega))}\Bigg).
\end{align*}
Since $f$ satisfies \eqref{ass:f_uniform_continuity}, Lebesgue's dominated convergence theorem implies 
\begin{align}
\lim_{n\to\infty}E\left[\left\vert\int_0^t\int_H \big( T_f(X_n(s-),h)-T_f(X(s-),h)\big)\,\mu^{X_n}({\rm d}s,{\rm d}h)\right\vert\right]
    =0.
\end{align}
For the next step, fix  $\epsilon, \epsilon^\prime>0$, and use for any $m,n\in\mathbb{N}$ the decomposition 
\begin{align}
&\left(\left\vert\int_0^t\int_HT_f(X(s-),h)\left(\mu^{X_n}({\rm d}s,{\rm d}h)-\mu^{X}({\rm d}s,{\rm d}h)\right)\right\vert>\epsilon\right)\notag\\
 &\qquad\leq P\left(\left\vert\int_0^t\int_{\overline{B}_H(1/m)}T_f(X(s-),h)\,\mu^{X_n}({\rm d}s,{\rm d}h)\right\vert>\frac{\epsilon}{3}\right)\notag\\
&\qquad\qquad\qquad+P\left(\left\vert\int_0^t\int_{\overline{B}_H(1/m)}T_f(X(s-),h)\, \mu^{X}({\rm d}s,{\rm d}h)\right\vert>\frac{\epsilon}{3}\right)\notag\\
    &\qquad\qquad\qquad +P\left(\left\vert\int_0^t\int_{\overline{B}_H(1/m)^c}T_f(X(s-),h)\, \left(\mu^{X_n}({\rm d}s,{\rm d}h)-\mu^{X}({\rm d}s,{\rm d}h)\right)\right\vert>\frac{\epsilon}{3}\right).
    \label{proof:convergence_technical0}
\end{align}
Since Taylor's remainder theorem in the integral form implies $|T_f(X(s-),h)|\leq \frac{1}{2}\norm{D^2f}_\infty\norm{h}^2$ for all $h\in H$, 
we obtain by applying Theorem \ref{th.compensator} and Inequality \eqref{estimate_alpha_norm} that
\begin{align*}
&    \E\left[\left\vert\int_0^t\int_{\overline{B}_H(1/m)}T_f(X(s-),h)\, \mu^{X_n}({\rm d}s,{\rm d}h)\right\vert\right]\\
    &\qquad \le \frac{1}{2}\norm{D^2f}_\infty\E\left[\int_0^t\int_{\overline{B}_H(1/m)}\norm{h}^2\, \left(\lambda\circ \left(R_nG(X_n(s))\right)^{-1}\right)({\rm d}h){\rm d}s\right]
 \leq d_\alpha^mK_G^\alpha\frac{1}{2}\norm{D^2f}_\infty T.
\end{align*}
Since the last line is independent of $n\in\mathbb{N}$ and $d_\alpha^m\to 0$ as $m\to\infty$ according to Inequality \eqref{estimate_alpha_norm},  Markov's inequality implies that there exists $m_1\in\N$ such that 
for all $m\ge m_1$ and all $n\in \N$ 
\begin{align}\label{proof:convergence_technical2}
P\left(\left\vert\int_0^t\int_{\overline{B}_H(1/m)} T_f(X(s-),h)\,\mu^{X_n}({\rm d}s,{\rm d}h)\right \vert >\frac{\epsilon}{3}\right)\le \epsilon^\prime. 
\end{align}
Exactly the same arguments establish that for all $m\ge m_1$
\begin{align}\label{proof:convergence_technical3}
   P\left(\left\vert\int_0^t\int_{\overline{B}_H(1/m)}T_f(X(s-),h)\,\mu^{X}({\rm d}s,{\rm d}h)\right\vert>\frac{\epsilon}{3}\right)\le \epsilon^\prime. 
\end{align}
For the last term in \eqref{proof:convergence_technical0} we calculate for each $m,n\in\N$ that
\begin{align}\label{proof:convergence_technical3b}
    &    \int_0^t\int_{\overline{B}_H(1/m)^c}T_f(X(s-),h)\, \left(\mu^{X_n}({\rm d}s, {\rm d}h)-\mu^X({\rm d}s, {\rm d}h)\right)\notag \\
    &=\sum_{0\leq s\leq t} T_f(X(s-),\Delta X_n(s))\mathbbm{1}_{\overline{B}_H(1/m)^c}({\Delta X_n(s)})\notag\\
    &\qquad  -\sum_{0\leq s\leq t} T_f(X(s-),\Delta X(s))\mathbbm{1}_{\overline{B}_H(1/m)^c}({\Delta X(s)})\notag\\
    &=    \sum_{0\leq s\leq t}T_f(X(s-),\Delta X_n(s))\left(\mathbbm{1}_{\overline{B}_H(1/m)^c}({\Delta X_n(s)})-\mathbbm{1}_{\overline{B}_H(1/m)^c}({\Delta X(s)})\right)\notag\\
    &\qquad+ \sum_{0\leq s\leq t}\big(T_f(X(s-),\Delta X_n(s))-T_f(X(s-),\Delta X(s))\big)\mathbbm{1}_{\overline{B}_H(1/m)^c}({\Delta X(s)}).
\end{align}
For estimating the first term in the last line, we use 
the equality $\mathbbm{1}_A(x)-\mathbbm{1}_{A}(y)=\mathbbm{1}_A(x)\mathbbm{1}_{A^c}(y)-\mathbbm{1}_{A^c}(x)\mathbbm{1}_A(y)$. 
For the first term, resulting from the application of this identity, we conclude from  Taylor's remainder theorem in the integral form that 
\begin{align}\label{proof:convergence_technical30e}
    &\left\vert\sum_{0\leq s\leq t} T_f\left(X(s-\right),\Delta X_n(s))\mathbbm{1}_{\overline{B}_H(1/m)^c}(\Delta X_n(s))\mathbbm{1}_{\overline{B}_H(1/m)}(\Delta X(s))\right\vert\notag\\
    & \leq \sum_{0\leq s\leq t}\big|T_f(X(s-),\Delta X_n(s))\big|\mathbbm{1}_{\overline{B}_H(1/m)^c}(\Delta X_n(s))\mathbbm{1}_{\overline{B}_H(1/m)}(\Delta X(s))\notag\\
    & \leq\frac12\norm{D^2f}_{\infty}\sum_{0\leq s\leq t}\norm{\Delta X_n(s)}^2\mathbbm{1}_{\overline{B}_H(1/m)^c}(\Delta X_n(s))\mathbbm{1}_{\overline{B}_H(1/m)}(\Delta X(s))\notag\\
    &\leq c_f \!\!\sum_{0\leq s\leq t}\!\!\left(\norm{\Delta X(s)}^2+\norm{\Delta X_n(s)-\Delta X(s)}^2\right)\mathbbm{1}_{\overline{B}_H(1/m)^c}(\Delta X_n(s))\mathbbm{1}_{\overline{B}_H(1/m)}(\Delta X(s)),
\end{align}
where we used the notation $c_f:=\norm{D^2f}_{\infty}$. 
Applying Theorem \ref{th.compensator}  and using the boundedness assumption on $G$ in (A3) result in  
\begin{align*}
&   \E\left[\sum_{0\leq s\leq t}\norm{\Delta X(s)}^2\mathbbm{1}_{\overline{B}_H(1/m)}(\Delta X(s))\right]
    =\E\left[\int_0^t\int_{\overline{B}_H(1/m)}\norm{h}^2\, \mu^X({\rm d}s, {\rm d}h)\right]\\
&\qquad     =\E\left[\int_0^t\int_{\overline{B}_H(1/m)}\norm{h}^2\,\big(\lambda\circ G(X(s-))^{-1}\big)({\rm d}h, {\rm d}s)\right]
    \leq d_\alpha^mTK_G^\alpha.
\end{align*}
Since $d_\alpha^m\to 0$ as $m\to\infty$, Markov's inequality implies that there exists $m_2\in\N$ with $m_2 \ge m_1$ such that 
for all $m\ge m_2$ and all $n\in \N$ 
\begin{align}
P\left(\sum_{0\leq s\leq t}\norm{\Delta X(s)}^2 \mathbbm{1}_{\overline{B}_H(1/m)^c}(\Delta X_n(s))\mathbbm{1}_{\overline{B}_H(1/m)}({\Delta X(s)})\ge \frac{\epsilon}{24c_f}\right)\le \frac{\epsilon^\prime}{8}.
    \label{proof:convergence_technical3e}
\end{align}
In the remaining part of the proof fix $m\ge m_2 \ge m_1$  such that \eqref{proof:convergence_technical3e} is satisfied.
There exists $n_1\in \N$ such that the set $A_n:=\{\sup_{s\in [0,t]}\norm{\Delta X_n(s)-\Delta X(s)}\le \tfrac{1}{2m} \}$ satisfies $P(A_n^c)\le \tfrac{\epsilon^\prime}{16} $ for all $n\ge n_1$. If $\sup_{s\in [0,t]}\norm{\Delta X_n(s)}\ge \tfrac1m$ then we obtain on $A_n$ for $n\ge n_1$ that 
\begin{align*}
  \sup_{s\in [0,t]}\norm{\Delta X(s)}\ge - \sup_{s\in [0,t]}\norm{\Delta X(s)- \Delta X_n(s)} +\sup_{s \in[0,t]} \norm{\Delta X_n(s)}\ge \frac{1}{2m}. 
\end{align*}
Consequently, we obtain for all $n\ge n_1$ that 
\begin{align*}
& P\left(\sum_{0\leq s\leq t}\norm{\Delta X_n(s)-\Delta X(s)}^2\mathbbm{1}_{\overline{B}_H(1/m)^c}(\Delta X_n(s))\mathbbm{1}_{\overline{B}_H(1/m)}(\Delta X(s))\ge \frac{\epsilon}{24c_f}\right)\\
&\qquad \leq  P\left(\sum_{0\leq s\leq t}\norm{\Delta X_n(s)-\Delta X(s)}^2\mathbbm{1}_{\overline{B}_H(1/2m)^c}(\Delta X(s))\ge \frac{\epsilon}{24c_f}\right)+\frac{\epsilon^\prime}{16}. 
\end{align*}
Since $X$ has only finitely many jumps in $\overline{B}_H(1/2m)^c$ on $[0,t]$ and $\Delta X_n(s)$ converges to $\Delta X(s)$ 
for all $s\in [0,t]$, there exists $n_2$ such that for all $n\ge n_2$ 
\begin{align*}
    P\left(\sum_{0\leq s\leq t}\norm{\Delta X_n(s)-\Delta X(s)}^2\mathbbm{1}_{\overline{B}_H(1/m)^c}(\Delta X_n(s))\mathbbm{1}_{\overline{B}_H(1/m)}(\Delta X(s)) \ge \frac{\epsilon}{24c_f}\right)\le \frac{\epsilon^\prime}{8}.
\end{align*}
Applying this together  with \eqref{proof:convergence_technical3e} to Inequality \eqref{proof:convergence_technical30e} proves that for $m\ge m_2$ there exists $ n_2$ such that for  all $n\ge n_2$ 
\begin{align}
    P\left( \sum_{0\leq s\leq t} T_f\left(X(s-\right),\Delta X_n(s))\mathbbm{1}_{\overline{B}_H(1/m)^c}(\Delta X_n(s))\mathbbm{1}_{\overline{B}_H(1/m)}(\Delta X(s)) \ge \frac{\epsilon}{12}\right)\le \frac{\epsilon^\prime}{4}.\quad
    \label{proof:convergence_technical3d}
\end{align}
As $\Delta X_n$ converges to $\Delta X$ uniformly on $[0,T]$ almost surely we obtain that for almost all $\omega\in\Omega$ we have
\begin{align*}
    \mathbbm{1}_{\overline{B}_H(1/m)}(\Delta X_n(s)(\omega))\mathbbm{1}_{\overline{B}_H(1/m)^c}(\Delta X(s)(\omega))=0
\end{align*}
if $n$ is large enough. Therefore
\begin{align*}
    \lim_{n\to\infty}\left(\sum_{0\leq s\leq t}T_f(X(s-), \Delta X_n(s))\mathbbm{1}_{\overline{B}_H(1/m)}(\Delta X_n(s))\mathbbm{1}_{\overline{B}_H(1/m)^c}(\Delta X(s))\right)=0 \quad \text{a.s.}
\end{align*}
and thus we obtain that there exists $n_3$ such that for all $n\ge n_3$ we have
\begin{align*}
P\left(  \sum_{0\leq s\leq t}T_f(X(s-), \Delta X_n(s))\mathbbm{1}_{\overline{B}_H(1/m)}(\Delta X_n(s))\mathbbm{1}_{\overline{B}_H(1/m)^c}(\Delta X(s))\ge \frac{\epsilon}{12}\right) \le \frac{\epsilon^\prime}{4}. 
\end{align*}
Combining this with \eqref{proof:convergence_technical3d} shows  that for $m\ge m_2$ 
and $n\ge \max\{n_2,n_3\}$ we have
\begin{align}   \label{proof:convergence_technical3c}
P\left(\sum_{0\leq s\leq t}T_f(X(s-),\Delta X_n(s))\left(\mathbbm{1}_{\overline{B}_H(1/m)^c}({\Delta X_n(s)})-\mathbbm{1}_{\overline{B}_H(1/m)^c}({\Delta X(s)})\right)\ge \frac{\epsilon}{6}\right)\le \frac{\epsilon^\prime}{2}. 
\end{align}
Since $X$ has only finitely many jumps in $\overline{B}_H(1/m)^c$ on $[0,t]$ and $\Delta X_n(s)$ converges to $\Delta X(s)$ 
for all $s\in [0,t]$, there exits $n_4$ such that for all $n\ge n_4$
\begin{align*}
P\left( \sum_{0\leq s\leq t}\big(T_f(X(s-),\Delta X_n(s))-T_f(X(s-),\Delta X(s))\big)\mathbbm{1}_{\overline{B}_H(1/m)^c}({\Delta X(s)})\ge \frac{\epsilon}{6}\right)\le \frac{\epsilon^\prime}{2}. 
\end{align*}
Applying this together with \eqref{proof:convergence_technical3c} to \eqref{proof:convergence_technical3b} shows
\begin{align} \label{eq.last_estimate}
P\left(\int_0^t\int_{\overline{B}_H(1/m)^c}T_f(X(s-),h)\,\left(\mu^{X_n}({\rm d}s,{\rm d}h)-\mu^{X}({\rm d}s,{\rm d}h)\right\vert \ge \frac{\epsilon}{3}\right)\le \epsilon^\prime.
\end{align}
By applying Equations \eqref{proof:convergence_technical2},\eqref{proof:convergence_technical3} and \eqref{eq.last_estimate} to \eqref{proof:convergence_technical0}, the proof is now complete. 
\end{proof}

\begin{proof}[Proof of Theorem \ref{prop:ito_formula_mild}]
Let $(X_n)_{n\in\N}$ be the solutions to \eqref{eq:equation_approx}. According to Corollary \ref{le.yosida_convergence_uniform},  we can assume that $(X_n)_{n\in\N}$ converges both in $\mathcal{C}([0,T], L^p(\Omega,H))$ and uniformly on $[0,T]$ almost surely to the mild solution $X$, which has c\`adl\`ag paths.
Since $X_n$ is a strong solution to \eqref{eq:equation_approx} according to Proposition \ref{pro.weak-is-strong}, 
the Itô formula in Theorem  \ref{thm:ito_formula} implies for all $t\geq 0$ and $n\in\mathbb{N}$ that 
\begin{align}\label{proof:ito_mild2}
    f(X_n(t))&=
f(R_nx_0)+\int_0^t\langle G(X_n(s-))^*R_n^*Df(X_n(s-)),\cdot\rangle \,{\rm d}L(s)\notag\\
&\qquad +\int_0^t\langle Df(X_n(s)), AX_n(s) \rangle \,{\rm d}s +\int_0^t\langle Df(X_n(s)), R_nF(X_n(s)) \rangle \,{\rm d}s\\
    &\qquad +\int_0^t\int_H\bigg(f(X_n(s-)+h)-f(X_n(s-))-\langle Df(X_n(s-)),h\rangle\bigg) \, \mu^{X_n}({\rm d}s, {\rm d}h) .\notag
\end{align}
 Continuity of $f$ shows that 
$f(X_n(t))\to f(X(t))$ and  $f(R_nx_0)\to f(x_0)$ a.s.  Inequality \eqref{eq:moment_estimate_integral} implies for the first integral in \eqref{proof:ito_mild2} that 
\begin{align*}
    \E&\left[\norm{\int_0^t\langle G(X_n(s-))^*R_n^*Df(X_n(s-)) ,\cdot\rangle \,{\rm d}L(s)-\int_0^t\langle G(X(s-))^*Df(X(s-)) ,\cdot\rangle\, {\rm d}L(s)}\right]\\
    &\leq e_{1,\alpha}\left(\E\left[\int_0^t\norm{G(X_n(s))^*R_n^*Df(X_n(s))-G(X(s))^*Df(X(s))}_{\mathcal{L}_2(U,\mathbb{R})}^\alpha\, {\rm d}s\right]\right)^{1/\alpha},
\end{align*}
which tends to zero by a similar argument as in the proof of Lemma \ref{le.generator_convergence}. It follows in $L_P^1(\Omega,\mathbb{R})$ that 
\begin{align*}
\lim_{n\to\infty} \int_0^t\langle G(X_n(s-))^*R_n^*Df(X_n(s-)),\cdot\rangle \,{\rm d}L(s)=\int_0^t\langle G(X(s-))^*Df(X(s-)) ,\cdot\rangle \,{\rm d}L(s). 
\end{align*}
 Lemma \ref{le.convergence_technical} implies in $L_P^0(\Omega,\mathbb{R})$ that 
\begin{align*}
& \lim_{n\to\infty} \int_0^t\int_H \big(f(X_n(s-)+h)-f(X_n(s-))-\langle Df(X_n(s-)),h\rangle\big)\, \mu^{X_n}({\rm d}s, {\rm d}h)\\
    &\qquad =\int_0^t\int_H\big(f(X(s-)+h)-f(X(s-))-\langle Df(X(s-)),h\rangle \big)\, \mu^{X}({\rm d}s, {\rm d}h).  
\end{align*}
Convergence of $(X_n)_{n\in\N}$ and Lebesgue's dominated convergence theorem yields 
\begin{align*}
  \lim_{n\to\infty} \int_0^t\langle Df(X_n(s) ), R_nF(X_n(s)) \rangle \,{\rm d}s =\int_0^t\langle Df(X(s)),F(X(s)) \rangle\,{\rm d}s \quad \text{a.s.}
\end{align*}
As all terms in \eqref{proof:ito_mild2} converge in $L_P^0(\Omega,\mathbb{R})$, it follows that the remaining term 
\begin{align*}
    \int_0^t\langle Df(X_n(s)), AX_n(s) \,\rangle{\rm d}s
\end{align*}
also converges in $L_P^0(\Omega,\mathbb{R})$, which completes the proof.
\end{proof}

\section{Appendix}

\begin{Lemma}\label{le.uniform_convergence}
Let $V$ be a separable Hilbert space with the norm $\norm{\cdot}_V$ and let $A_m\in\mathcal{L}(V)$ be a sequence of operators converging strongly to $0$. If $(B_n)_{n\in\mathbb{N}}$ is a tight sequence of uniformly bounded $V$-valued random variables, then it follows for all $p>0$ that
\[\lim_{m \rightarrow \infty}\sup_{n \in \mathbb{N}} \E\left[\norm{A_mB_n}_V^p\right]=0.\]
\end{Lemma}

\begin{proof}
Let $\epsilon>0$ be fixed. Our assumptions guarantee that there exists a constant $c>0$ such that 
$\sup_{n,m \in \mathbb{N}} \norm{A_m B_n}_V^p\leq c$ a.s.\ and  a compact set $K_\epsilon \subseteq V$ satisfying $P(B_n \notin K_\epsilon)<\frac{\epsilon}{c}$ for every $n\in\mathbb{N}$. Since continuous mapping converging to zero converges uniformly on compacts there exists $m_1 \in \mathbbm{N}$ such that for all $m\geq m_1$ we have
\begin{equation*} 
    \sup_{n \in \mathbb{N}} \int_{\{B_n \in K_\epsilon\}}\norm{A_mB_n(\omega)}_V^p\, P({\rm d}\omega)<\epsilon.
\end{equation*}
It follows for all $n\in\N$ and $m\ge m_1$ that 
\begin{align*}
 \E\left[\norm{A_mB_n}_V^p\right]\leq  \int_{\{B_n \in K_\epsilon\}}\norm{A_mB_n(\omega)}_V^p\, P({\rm d}\omega)+ \int_{\{B_n \notin K_\epsilon\}}\norm{A_mB_n(\omega)}_V^p\, P({\rm d}\omega)\le 2 \epsilon, 
\end{align*}
which completes the proof. 
\end{proof}

\end{document}